\documentclass[reqno,final]{amsart}

\usepackage{caption}
\usepackage{url}
\usepackage{amsfonts,latexsym,graphicx,amsmath,amssymb,bbm,enumitem,citesort}
\usepackage{enumitem,mathabx} 
\usepackage{subfig}
\makeindex

\usepackage[color,notref,notcite]{showkeys}

\definecolor{labelkey}{rgb}{0.6,0,1}

\allowdisplaybreaks

\definecolor{labelkey}{rgb}{0.6,0,1}

\newcommand{\mathbi}[1]{{\boldsymbol #1}}

\newcommand{\eop}{{\unskip\nobreak\hfil\penalty50
           \hskip2em\hbox{}\nobreak\hfil\mbox{\rule{1ex}{1ex} \qquad}
   \parfillskip=0pt
   \finalhyphendemerits=0\par\medskip}}
\renewenvironment{proof}[1][]{\noindent {\bf Proof#1. } }{\eop}

\newtheorem{theorem}{Theorem}[section]
\newtheorem{remark}[theorem]{Remark}
\newtheorem{lemma}[theorem]{Lemma} 
\newtheorem{definition}[theorem]{Definition}
\newtheorem{algorithm}[theorem]{Algorithm}
\newtheorem{corollary}[theorem]{Corollary}

\definecolor{shadecolor}{gray}{0.92}
\definecolor{TFFrameColor}{gray}{0.92}
\definecolor{TFTitleColor}{rgb}{0,0,0}

\newcommand{\ba}{\begin{array}{llll}   }
\newcommand{\bac}{\begin{array}{c}}
\newcommand{\bari}{\begin{array}{r}}
\newcommand{\ea}{\end{array}}

\newcommand{\ban}{\begin{array}{llll}}
\newcommand{\ean}{\end{array}}

\newcommand{\be}{\begin{equation}}
\newcommand{\ee}{\end{equation}}

\newcommand{\beqsys }{\beqtab \left \{ \begin{array}{l}}
\newcommand{\eeqsys }{\end{array} \right . \eeqtab }

\newcommand{\benum}{\begin{enumerate}}
\newcommand{\eenum}{\end{enumerate}}

\newcommand{\beqtab}{\begin{eqnarray}} 
\newcommand{\eeqtab}{\end{eqnarray}}

\newcommand{\mD}{{\mathcal D}}
\newcommand{\mI}{{\mathcal I}}
\newcommand{\mId}{\mathbbm{1}}

\newcommand{\bxi}{\mathbi{\xi}}

\renewcommand{\d}{{\rm d}}

\newcommand{\disc}{{\mathcal D}}

\newcommand{\dive}{{\rm div}}
\renewcommand{\div}{{\mathop{\rm div}}}

\newcommand{\N}{\mathbb N}

\newcommand{\R}{\mathbb R}

\newcommand{\x}{\mathbi{x}}

\newcommand{\vphi}{\mathbi{\phi}}

\usepackage{xparse}
\DeclareDocumentCommand{\RPiD}{ O{\disc} O{,0} }{\Pi_{#1}(X_{#1#2})}

\def\Fdof#1{{\bm{\mathcal{F}}(#1,\R)}}
\makeatletter
\def\Fdof{\@ifnextchar[{\@with}{\@without}}
\def\@with[#1]#2{{\bm{\mathcal{F}}(#2;#1)}}
\def\@without#1{{\bm{\mathcal{F}}(#1,\R)}}
\makeatother

\def\RT0{\mathbb{RT}_0}

\def\ms{\widehat{m}}

\newcommand{\s}{\mathsf{s}}
\newcommand{\f}{\mathsf{f}}
\newcommand{\ii}{\mathsf{i}}

\newcommand{\eminus}{\mbox{\hspace*{0.1pt}e-}}

\begin{document}

	\title[]{The gradient discretisation method for slow and fast diffusion porous media equations}
	
\author{J\'er\^ome Droniou}
\address{School of Mathematics, Monash University, Clayton, Victoria 3800, Australia.
\texttt{jerome.droniou@monash.edu}}
\author{Kim-Ngan Le}
\address{School of Mathematics, Monash University, Clayton, Victoria 3800, Australia.
\texttt{ngan.le@monash.edu}}

	\date{\today}
	
	%
	%
	
	\subjclass[2010]{65M08, 
	65M12, 
	65M60, 
    76S05}	

	\maketitle
\begin{abstract}
	The gradient discretisation method (GDM) is a generic framework for designing and analysing numerical schemes for diffusion models. In this paper, we study the GDM for the porous medium equation, including fast diffusion and slow diffusion models, and a concentration-dependent diffusion tensor. Using discrete functional analysis techniques, we establish a strong $L^2$-convergence of the approximate gradients and a uniform-in-time convergence for the approximate solution, without assuming non-physical regularity assumptions on the data or continuous solution. Being established in the generic GDM framework, these results apply to a variety of numerical methods, such as finite volume, (mass-lumped) finite elements, etc. The theoretical results are illustrated, in both fast and slow diffusion regimes, by numerical tests based on two methods that fit the GDM framework: mass-lumped conforming $\mathbb{P}_1$ finite elements and the Hybrid Mimetic Mixed method.
\end{abstract}
\section{Introduction}
In this paper, we study nonlinear porous media equations of the form
\begin{align}\label{eq: PME}
&\partial_t u - \div\left[\Lambda(\cdot,u)\nabla \beta(u)\right] = f\quad\text{in } \Omega_T:=(0,T)\times \Omega,\nonumber\\
&u(0,\cdot) = u_0 \quad\text{in }  \Omega,\\
&\beta(u) = 0 \quad\text{on } (0,T)\times \partial\Omega,\nonumber
\end{align}
where $\beta(u)=|u|^{m-1}u$ with $m>0$, $\Omega$ is an open bounded domain of $\R^d$ ($d\geq 1$) with boundary $\partial \Omega$, $f\in L^2(\Omega_T)$, $u_0\in L^{m+1}(\Omega)$, and  $\Lambda:\Omega\times \R\to\mathcal S^d(\R)$ is a symmetric tensor which is measurable with respect to its first variable, continuous with respect to its second variable, and is uniformly elliptic and bounded:
\begin{equation}\label{eq:hyp.Lambda}
\exists\underline{\lambda},\overline{\lambda}\in(0,\infty)\mbox{ s.t. }
\underline{\lambda}|\bxi|^2\le \Lambda(\x,s)\bxi\cdot\bxi\le \overline{\lambda}|\bxi|^2\quad
\forall s\in\R\,,\forall\bxi\in\R^d\,,\mbox{ for a.e. $\x\in\Omega$}.
\end{equation}

In the case $\Lambda={\rm Id}$ and $m>1$, the equation~\eqref{eq: PME} is the standard model of diffusion of a gas in porous media, which is also called the slow diffusion model. The case $m=2$ describes the flow of an ideal gas in porous media while $m> 2$ that of a diffusion of a compressible fluid through porous media. Other choices of exponent appear in different physical situations, such as thermal propagation in plasma ($m=6$) or plasma radiation ($m=4$). The slow diffusion equation is not uniformly elliptic as it degenerates at unknown points where $u=0$. An interesting feature is that, if $u_0$ is compactly supported, then at any $t>0$ the support of $u(t)$ has a free boundary with a finite speed of propagation, see e.g.~\cite{Caffarelli1980,Vazquez1992}. Equation \eqref{eq: PME} with $\Lambda$ depending on $u$ appears in Richards' model, when the relative permeability depends on the hydraulic charge, and in diphasic models from petroleum reservoir simulation, when the capillary pressure follows a power law \cite{aziz2002petroleum}.

In the fast diffusion case $0<m<1$, the equation~\eqref{eq: PME} is relevant in the description of plasma physics, the kinetic theory of gas or fluid transportation in porous media~\cite{BerrymanHolland1978,CarChaGraSwi1990,Vaz2006}. Since the modulus of ellipticity $|u|^{m-1}$ blows up whenever $u$ vanishes, the fast diffusion equation is a singular equation. An essential difference between the fast and slow diffusion cases is that for fast diffusion the solution decays to zero in some finite time depending on the initial data while for slow diffusion the solution decays to zero in infinite time like an inverse power of $t$, see e.g.~\cite{Berryman1980,Diben1991}.

There is a vast literature on the numerical approximation of \eqref{eq: PME} (mostly with $\Lambda={\rm Id}$ or not depending on $u$), possibly in a recast form. We only mention a few relevant studies here. In \cite{MNV87}, the authors consider the time discretisation of $\partial_t u-\Delta\beta(u)=f(\beta(u))$, using a maximal-monotone operator approach. Error estimates are obtained in $L^\infty(0,T;H^{-1}(\Omega))$ and $L^2(\Omega_T)$ norms, with a rate depending on the smoothness of the initial condition and under the assumption that $\beta$ is globally Lipschitz-continuous.
\cite{AWZ97} considers the mixed finite element approximation of a degenerate parabolic equation with advection, arising in petroleum simulations. Error estimates, in similar norms as above, are obtained under the assumption that the non-linearities are Lipschitz-continuous on the entire range of the numerical solutions. The Lipschitz-continuity assumption only covers \eqref{eq: PME} in the slow diffusion case $m>1$, provided a uniform bound is known or assumed on the approximate solution. The case of systems of PDEs modelled on \eqref{eq: PME} is analysed in \cite{JK91}, in which the $L^2$-convergence of a semi-discrete scheme (with relaxation of the non-linear source term) is established in the case $m>1$. In a recent work~\cite{Teso2018}, a general nonlinear diffusion equation on the whole space $\R^d$ is considered. The authors propose monotone schemes of finite difference type on Cartesian meshes and obtain $L^1_{\text{loc}}$-convergence for these schemes, using an approach in which this convergence follows from generic estimate results on perturbed version of \eqref{eq: PME}.

If we formally represent this equation as 
\begin{equation}\label{eq: PME2}
\partial_t u - \dive(\beta'(u)\nabla u) = f,
\end{equation}
where $\beta'(u) = m|u|^{m-1}$, then we obtain a classical nonlinear diffusion equation. The term $|u|^{m-1}$ in~\eqref{eq: PME2} induces the degeneracy that raises many challenges in the analysis of the porous media equations. These problems have been studied extensively both in theory, see e.g.~\cite{Berryman1980,Vazquez2007,Vaz2006}, numerical analysis, see e.g.~\cite{Duque2013,Ebmeyer1998,Ebmeyer2008,Rui2017}, as well as in numerical approximations (without convergence proof)~\cite{Li2018}. 
In particular, authors in~\cite{Duque2013} study equation~\eqref{eq: PME2} with variable exponent of nonlinearity, i.e. $\beta'(u)$ is replaced by $|u|^{\gamma(x)}$ where $\gamma >1$ (i.e. $m>2$ in our case). In order to deal with the degeneracy in the problem, an approximate regularized problem is investigated.  A space-time discretization scheme using the finite element method in space and the discontinuous Galerkin method in time  is proposed for the regularized model (not the original problem). Furthermore, error estimates are obtained with strong regularity assumptions on the solution of the regularized model, which are not expected to be satisfied by classical solutions to \eqref{eq: PME} such as Barenblatt solutions~\cite{Barenblatt52,Pattle1959}. 
A space and time dependent exponent $\gamma$ is studied in~\cite{Rui2017} using the same method as in~\cite{Duque2013}.
The slow diffusion case ($m>1$) is studied in~\cite{Ebmeyer1998} where a fully discrete scheme is proposed and $L^2$ error estimates are proved with strong assumptions on the solution and the pressure $\beta'(u)$. 

Another common way to design and analyse numerical methods for \eqref{eq: PME} is to re-cast it, using the Kirchhoff transform:
 \begin{equation}\label{eq:Richards}
 \partial_t \psi(w) - \Delta w=0,
 \end{equation}
where $w:= |u|^{m-1}u$ and $u=\psi(w):= |w|^{(1-m)/m}s$. 
 Finite elements approximations of \eqref{eq:Richards} have been extensively studied in the literature. In \cite{NochettoVerdi88}, the conforming approximation is analysed in the case of a convex domain, using a smoothed version of $\psi$ and under assumptions only satisfied for $m>1$; error estimates, also accounting for the smoothing parameter, are obtained in norms similar to \cite{MNV87} and yield an $\mathcal O(h^{(m+1)/2m})$ rate for an optimal choice of the smoothing parameter and time step $\sim h^{(m+1)/m}$. Estimates in $L^{m+1}(\Omega_T)$ norm are also obtained, at least in dimension one, leading to an $h^{4m/(m+1)(3m-1)}$ convergence rate. We note in passing that the results of \cite{Nochetto86} allows, under a non-degeneracy property of the continuous solutions, to transform $L^p$-error estimates on the solutions into error estimates on the location of the free boundary. Mixed finite elements for \eqref{eq:Richards} have been considered in the following works. Under a Lipschitz-continuity assumption on $\psi$ (which, as noticed above, is not satisfied by \eqref{eq: PME}), \cite{EKR04} obtains error estimates for the Raviart--Thomas mixed finite elements in norms similar to \cite{MNV87,AWZ97}. The convergence of a linearisation scheme is established in \cite{PRK04} for a globally Lipschitz continuous $\psi$, and extended in \cite{Raduetc17} to non-Lipschitz functions. This assumption is relaxed in \cite{RPK08}, which considers the mixed finite element approximation of \eqref{eq:Richards} allowing degeneracy/singularity of $\psi$ which covers the whole range of $m$ in \eqref{eq: PME}; error estimates are obtained in uniform-in-time weak-in-space and averaged time-space norms as in \cite{AWZ97}, some of them under regularity assumptions on the flux variable. \cite{EHV06} proposes and analyses a scheme for \eqref{eq:Richards}, with an added advection term and $\psi$ globally Lipschitz; the diffusion component is handled by a non-conforming or mixed finite element, while the advection and reaction terms are discretised using centred finite volume. The Lipschitz-continuity assumption on $\psi$ can be also be relaxed by recasting \eqref{eq:Richards} in a form $\partial_t b(v)-\Delta c(v)=0$ that covers all the range $m\in (0,\infty)$ in \eqref{eq: PME} with globally Lipschitz continuous functions $b$ and $c$, but at the expense of a doubly degenerate equation and the usage of relaxation functions to control each fast and slow diffusion cases; this is the approach chosen in \cite{JK95,K97}, in which the convergence of as semi-discrete scheme is established. More recently, \cite{Ebmeyer2008} considered the fast and slow diffusion cases in which a fully discrete Galerkin approximation is considered for \eqref{eq:Richards}. Error estimates in non-standard quasi norms and rates of convergence are proved with strong regularity assumptions on the solution. 

Finally, let us mention that various numerical methods have been studied for other models of miscible or multi-phasic porous media flow, presented as systems of equations with one of them a non-linear parabolic equation, see e.g. \cite{Y97,eym03,CD07,RW11,EM12,GLR17,CR18}. In most cases, though, the parabolic equation is not really degenerate in the sense that its linearisation is uniformly parabolic.

In this paper, we use the gradient discretisation method to approximate~\eqref{eq: PME} and discrete functional analysis techniques to obtain an $L^\infty(0,T;L^{m+1}(\Omega))$-convergence result without assuming non-physical regularity assumptions on the data. The gradient discretisation method (GDM) is a generic framework for the design and analysis of numerical schemes for diffusion models. Using only a few discrete elements (a space, and a function and gradient reconstruction), it describes a variety of numerical schemes --- such as finite volumes, finite elements, discontinuous Galerkin, etc. --- and identifies three key properties of the discrete elements that ensure the convergence for linear and nonlinear models. We refer to the monograph \cite{Droniou.et.al2018} for a detailed presentation of the GDM, and of the methods it contains.

There are several advantages in using the GDM approach to analyse schemes for \eqref{eq: PME}. First, the convergence analysis readily applies to all schemes covered by the framework, which include many low- and high-order schemes; as noticed above, previous numerical analysis results on the porous medium equation seemed to mostly focus on conforming or mixed finite element discretisations. Second, the GDM provides general compactness results that can be used to simplify the analysis --- no need to establish \emph{ad hoc} compactness theorems for each specific method. Third, our approach seamlessly works for both slow ($m>1$) and fast ($m<1$) diffusion models, without having to re-cast the latter using the Kirchhoff transform, which is not the case of many analyses in the literature. Fourth, the uniform-in-time strong-in-space $L^\infty(0,T;L^{m+1}(\Omega))$ convergence result seems to be new even for schemes previously studied for \eqref{eq: PME} or \eqref{eq:Richards}.
A trade-off of using the GDM is that the analysis must be based on energy estimates, and cannot rely on maximum principles --- which are not satisfied by many schemes of practical interest, e.g. when dealing with polytopal meshes as encountered in applications.

The techniques we use are inspired by \cite{Droniou2016}, in which a doubly degenerate parabolic equation is considered. The convergence results in this reference cover \eqref{eq: PME} in the case where $\beta$ is globally Lipschitz continuous; this condition is critical to many elements of the analysis in \cite{Droniou2016}. Our contribution here is to consider the case where $\beta$ is not  uniformly Lipschitz continuous at infinity, or not Lipschitz-continuous at certain points; these two cases correspond to $m>1$ (porous medium equation) and $m<1$ (fast diffusion equation), respectively. As shown by Theorem \ref{theo:cv.GS}, considering a non-globally Lipschitz continuous nonlinearity changes the nature of the observed convergence; in particular, $L^2(\Omega)$ is no longer the natural space for the uniform-in-time convergence.

\medskip

Other approaches to convergence analysis of numerical schemes for \eqref{eq: PME} reside in using general Crandall--Liggett and Trotter--Kato approximation theories for monotone operators (see \cite{CL71}, \cite[Section 10.2.4]{Vazquez2007}, \cite{T58,K59} and \cite[Chapter 4]{IK02}). The Crandall--Liggett theorem deals with the time discretisation, while the Trotter--Kato theorem concerns the approximation of the spatial operator; upon establishing that one of these results is uniform with respect to the discretisation parameter of the other, the combination of these two yield a convergence for fully discrete schemes for the porous medium equation. This has however two limitations. Firstly, the $m$-accretivity of the porous medium only seems to be available in $L^1(\Omega)$, not in $L^p(\Omega)$ for $p>1$ \cite[Section 10.3]{Vazquez2007}; as a consequence, the convergence would only be $L^1$ in space. Secondly, the Crandall--Liggett and Trotter--Kato theories could only be applied for $\Lambda$ independent of $u$, since they require a monotone spatial operator. Our result, based on compactness and fine energy equality, yields on the contrary a uniform-in-time and $L^{m+1}$ in space convergence, and is applicable to non-monotone operators with $\Lambda$ depending on $u$.

\medskip

The paper is organised as follows. Section \ref{sec:gs} recalls notations of the GDM and introduce an implicit-in-time gradient scheme for~\eqref{eq: PME}. The definition of weak solutions and the main convergence result are stated in Section \ref{sec:main_result}. We provide a priori estimates for the approximate solution in Section \ref{sec:apriori}. Section \ref{sec:init_conv} contains the initial weak convergence of the gradient scheme. The main convergence result, including the uniform-in-time convergence, is proved in Section \ref{sec:unif_conv}. Numerical results are provided in Section \ref{sec:tests} to illustrate the generic convergence results; to this purpose, we choose two particular gradient schemes, the mass-lumped conforming $\mathbb{P}_1$ method and the Hybrid Mimetic Mixed method, and we evaluate their accuracies when approximating the Barenblatt solution. These tests are presented for a variety of exponents $m$, including both fast diffusion and slow diffusion ranges. A brief conclusion is presented in Section \ref{sec:concl}, and technical results are gathered in an appendix (Section \ref{sec:appen}).

\section{The gradient discretisation method and main convergence result}

\subsection{Gradient scheme}\label{sec:gs}
We recall here the notions of the gradient discretisation method. The idea of this general analysis framework is to replace, in the weak formulation of the problem, the continuous space and operators by discrete ones; the set of discrete space and operators is called a gradient discretisation (GD), and the scheme obtained after substituting these elements into the weak formulation is called a gradient scheme (GS). The convergence of the obtained GS can be established based on only a few general concepts on the underlying GD. Moreover, different GDs correspond to different classical schemes (finite elements, finite volumes, etc.). Hence, the analysis carried out in the GDM directly applies to all these schemes, and does not rely on the specificity of each particular method.

\begin{definition}\label{def: gdm}
$\mD=\bigl(X_{\mD,0},\Pi_\mD,\nabla_\mD,\mI_\mD,\bigl(t^{(n)}\bigr)_{n=0,\cdots,N}\bigr)$ is a space-time gradient discretisation for homogeneous Dirichlet boundary conditions, with piecewise constant reconstruction, if
\begin{enumerate}[label=(\roman*)]
\item the set of discrete unknowns $X_{\mD,0}$ is a finite dimensional real vector space,
\item\label{def:PiD} the linear map $\Pi_\mD:X_{\mD,0}\rightarrow L^\infty(\Omega)$ is a piecewise constant reconstruction operator in the following sense: there exists a basis $(\mathbf{e}_i)_{i\in I}$ of $X_{\mD,0}$ and a family $(\Omega_i)_{i\in I}$ of disjoint subsets of $\Omega$ such that, for all $u=\sum_{i\in I}u_i\mathbf{e}_i\in X_{\mD,0}$, it holds $\Pi_\mD u = \sum_{i\in I}u_i\mathbf{1}_{\Omega_i}$, where $\mathbf{1}_{\Omega_i}$ is the characteristic function of $\Omega_i$,
\item the linear mapping $\nabla_\mD: X_{\mD,0}\rightarrow L^2(\Omega)^d$ gives a reconstructed discrete gradient. It must be chosen such that $\|\nabla_\mD\cdot\|_{L^2(\Omega)}$ is a norm on $X_{\mD,0}$,
\item $\mI_\mD : L^{m+1}(\Omega)\rightarrow X_{\mD,0}$ is an interpolation operator,
\item  $t^{(0)}=0<t^{(1)}<\cdots<t^{(N)} = T$.  
\end{enumerate}
We then let $\delta t^{(n+\frac12)} = t^{(n+1)}-t^{(n)}$ and $\delta t_\mD = \max_{n=0,\cdots,N-1} \delta t^{(n+\frac12)} $.
\end{definition}

For any $\bigl(v^{(n)}\bigr)_{n=0,\cdots,N}\subset X_{\mD,0}$, we define the piecewise-constant-in-time functions $\Pi_\mD v:[0,T]\to L^\infty(\Omega)$, $\nabla_\mD v:(0,T]\to L^2(\Omega)^d$ and $\delta_\mD v:(0,T]\to L^2(\Omega)$ by:
For $n=0,\cdots,N-1$, for any  $t\in(t^{(n)},t^{(n+1)}]$, for a.e. $\x\in\Omega$
\begin{align*}
\Pi_\mD v(0,\x):=\Pi_\mD v^{(0)}(\x),\qquad
&\Pi_\mD v(t,\x):=\Pi_\mD v^{(n+1)}(\x),\\
\nabla_\mD v(t,\x):= \nabla_\mD v^{(n+1)}(\x),\qquad
&\delta_\mD v(t)=\delta_\mD^{(n+\frac12)} v:=\frac{\Pi_\mD v^{(n+1)} -\Pi_\mD v^{(n)}}{\delta t^{(n+\frac12)}}\in L^2(\Omega).
\end{align*}
Note that $\Pi_\mD v$ is defined everywhere, including at $t=0$. This will be required to state a uniform-in-time convergence result on this reconstructed function.

If $v=\sum_{i\in I}v_i\mathbf{e}_i\in X_{\mD,0}$ and $g:\R\to\R$ satisfies $g(0)=0$, we define $g(v):=\sum_{i\in I}g(v_i)\mathbf{e}_i\in X_{\mD,0}$. The piecewise constant feature of $\Pi_\mD$ then shows that
\begin{equation}\label{PiD:commute}
\forall v\in X_{\mD,0}\,,\quad \Pi_\mD g(v)=g(\Pi_\mD v).
\end{equation}

Once a GD has been chosen, an implicit-in-time gradient scheme for~\eqref{eq: PME} is defined the following way.

\begin{algorithm}[GS for \eqref{eq: PME}]
Set $u^{(0)}:=\mI_\mD u_0$ and let $u=\bigl(u^{(n)}\bigr)_{n=0,\cdots,N}\subset X_{\mD,0}$ satisfy:
\begin{align}\label{eq: gdmscheme}
\big\langle \delta_\mD u,\Pi_\mD \phi\big\rangle_{L^2(\Omega_T)}
+
\big\langle\Lambda(\cdot,\Pi_\mD u)\nabla_\mD\beta(u),  \nabla_\mD \phi\big\rangle_{L^2(\Omega_T)}
=
\big\langle f,  \Pi_\mD \phi\big\rangle_{L^2(\Omega_T)},
\end{align}
for all `test function' $\phi=\bigl(\phi^{(n)}\bigr)_{n=0,\cdots,N}\subset X_{\mD,0}$.
\end{algorithm}

In order to establish the stability and convergence of the GS~\eqref{eq: gdmscheme}, sequences of space-time gradient discretisations $(\mD_l)_{l\ge 1}$  are required to satisfy \textit{consistency, limit-conformity} and \textit{compactness} properties~\cite{Droniou.et.al2018}. The consistency is slightly adapted here to account for the nonlinearity we consider. In the following, we let $\ms=\max(1,1/m)$.

\begin{definition}[Consistency]
A sequence $(\mD_l)_{l\ge 1}$ of space-time gradient discretisations in the sense of Definition~\ref{def: gdm} is said to be consistent if
\begin{itemize}
\item for all $\phi\in L^{1+\ms}(\Omega)\cap H^1_0(\Omega)$, letting
\[
\hat{S}_{\mD_l}(\phi) := \min_{w\in X_{\mD_l}} 
\bigl(\|\Pi_{\mD_l} w - \phi\|_{L^{1+\ms}(\Omega)} + \|\nabla_{\mD_l} w-\nabla \phi\|_{L^2(\Omega)}\bigr),
\]
we have $\hat{S}_{\mD_l}(\phi)\rightarrow 0$ as $l\rightarrow \infty$,
\item  for all $\psi\in L^{m+1}(\Omega)$, $\Pi_{\mD_l}\mI_{\mD_l} \psi\rightarrow \psi$ in $ L^{m+1}(\Omega)$ as $l\rightarrow \infty$ 
\item $\delta t_{\mD_l} \rightarrow 0$ as $l\rightarrow \infty$.
\end{itemize}
\end{definition}
\begin{definition}[Limit-conformity]
A sequence $(\mD_l)_{l\ge 1}$ of space-time gradient discretisations in the sense of Definition~\ref{def: gdm} is said to be limit-conformity if, for all $\vphi\in H_\dive(\Omega):=\{\vphi\in  L^2(\Omega)^d\,:\,\div \vphi\in L^2(\Omega)\}$, letting 
\[
W_{\mD_l}(\vphi):=\max_{v\in X_{\mD_l}\backslash \{0\}}
\frac{\bigg|\displaystyle\int_\Omega \bigl(\nabla_{\mD_l}v(\x)\cdot \vphi(\x)+\Pi_{\mD_l}v(\x)\div\vphi(\x)\bigr)d\x\bigg|}{\|\nabla_{\mD_l}v\|_{L^2(\Omega)}},
\]
we have $W_{\mD_l}(\vphi)\rightarrow 0$ as $l\rightarrow\infty$.
\end{definition}
\begin{definition}[Compactness]
A sequence $(\mD_l)_{l\ge 1}$ of space-time gradient discretisations in the sense of Definition~\ref{def: gdm} is said to be 
compact if 
\[
\lim_{\bxi\rightarrow 0}\sup_{l\ge 1}\, T_{\mD_l}(\bxi) = 0,
\]
where 
\[
T_{\mD_l}(\bxi):=\max_{v\in X_{\mD_l}\backslash \{0\}}
\frac{\|\Pi_{\mD_l}v(\cdot + \bxi)-\Pi_{\mD_l}v\|_{L^2(\R^d)}}{\|\nabla_{\mD_l}v\|_{L^2(\Omega)}},\quad\forall \bxi\in\R^d,
\]
and $\Pi_{\mD_l}v$ has been extended by $0$ outside $\Omega$.
\end{definition}

A sequence of GDs that is compact or limit-conforming also satisfies another important property: the coercivity \cite[Lemmas 2.6 and 2.10]{Droniou.et.al2018}.
\begin{lemma}[Coercivity of sequences of GDs]
If a sequence $(\mD_l)_{l\ge 1}$ of space-time gradient discretisations in the sense of Definition~\ref{def: gdm} is compact or limit-conforming, then it is coercive: there exists a constant $C_p$ such that 
\[
C_{\mD_l} :=\max_{v\in X_{\mD_l}\backslash \{0\}}\frac{\|\Pi_{\mD_l}v\|_{L^2(\Omega)}}{\|\nabla_{\mD_l}v\|_{L^2(\Omega)}} \leq C_p,
\quad \forall l\ge 1.
\]
\end{lemma}

Finally, for a given gradient discretisation $\mD$, the following dual norm  on $\Pi_{\mD}(X_{\mD,0})\subset L^2(\Omega)$ will be used to obtain estimates on the discrete time derivative, key to establishing the strong compactness of the approximate solutions:
\begin{equation}\label{def:dualnorm}
\begin{aligned}
\forall v\in {}&\Pi_{\mD}(X_{\mD,0}),\\
&|v|_{*,\mD}:= \sup\bigg\{
\int_\Omega v(\x)\,\Pi_\mD \phi(\x)d\x \,\,: \,\,
\phi\in X_{\mD,0}\,,\; \|\nabla_\mD \phi\|_{L^2(\Omega)} = 1
\bigg\}.
\end{aligned}
\end{equation}

\subsection{The main convergence result}\label{sec:main_result}
To state the main result of this paper, we introduce a space of continuous functions for the weak topology of $L^{m+1}(\Omega)$, and we define a weak solution to \eqref{eq: PME}.
\begin{align*}
 C([0,T];L^{m+1}(\Omega)_{\mathrm{w}})
 := 
 &\text{ space of functions } v:[0,T]\rightarrow L^{m+1}(\Omega)\\
	&\text{ that are continuous for the weak topology of $L^{m+1}(\Omega)$}.
\end{align*}
For a given $s\in [0,T]$, we set $\Omega_s=(0,s)\times \Omega$. We also set 
\begin{equation}\label{def:zeta}
\zeta(z):=\int_0^z \beta(s)ds=\frac{1}{m+1}|z|^{m+1}\mbox{ for all $z\in\R$}.
\end{equation}

\begin{definition}[Weak solution to \eqref{eq: PME}]\label{def: weak sol}
Assume that $m>0$, $u_0\in L^{m+1}(\Omega)$ and $f\in L^2(\Omega_T)$. A weak solution to \eqref{eq: PME} is a function $\bar u$ such that
\begin{enumerate}[label=(\roman*)]
\item $\bar u\in C([0,T];L^{m+1}(\Omega)_{\mathrm{w}})$ and $\bar u(0,\cdot)=u_0$ in $L^{m+1}(\Omega)$,
\item $\beta(\bar u)\in L^2(0,T;H^1_0(\Omega))$, $\zeta(\bar u)\in L^{\infty}(0,T;L^1(\Omega))$, 
\item $\partial_t \bar u\in L^{2}(0,T;H^{-1}(\Omega))$, 
and for any 
$\phi\in L^2(0,T;H^1_0(\Omega))$
\begin{equation}\label{eq: weak sol}
\int_0^T\, _{H^{-1}}\langle \partial_t \bar u(t),\phi(t) \rangle_{H^1_0} \,dt
+
\langle \Lambda(\cdot,\bar{u})\nabla \beta(\bar u),\nabla \phi\rangle_{L^2(\Omega_T)}
=\langle f,\phi\rangle_{L^2(\Omega_T)}.
\end{equation}
\end{enumerate}
\end{definition}

The existence of a weak solution to \eqref{eq: PME} will be obtained as a by-product of our convergence analysis.

\begin{theorem}[Convergence of the gradient scheme]\label{theo:cv.GS}
Let $(\mD_l)_{l\ge 1}$ be a sequence of gradient discretisations that is consistent, limit-conforming and compact. Then, for each $l\ge 1$ there exists $u_l$ solution to the gradient scheme \eqref{eq: gdmscheme} with $\mD=\mD_l$. Moreover, there exists a  weak solution $\bar u$ to \eqref{eq: PME} in the sense of Definition~\ref{def: weak sol} such that, up to a subsequence as $l\to\infty$,
\begin{itemize}
\item $\Pi_{\mD_l}u_l\to \bar u$ strongly in $L^\infty(0,T;L^{m+1}(\Omega))$,
\item $\nabla_{\mD_l}\beta(u_l)\to \nabla \beta(\bar u)$ strongly in $L^2(\Omega_T)$.
\end{itemize}
\end{theorem}

\begin{remark}[Other nonlinearities]
The analysis that leads to this convergence result can be adapted to other forms of nonlinearities than $\beta(u)=|u|^{m-1}u$. We recall that the case of non-decreasing --possibly with plateaux-- nonlinearities that are globally Lipschitz-continuous is treated in \cite{Droniou2016}; for non-decreasing functions $\beta$ that are locally Lipschitz-continuous except at a finite number of points and/or at $\pm\infty$, the technique we develop below, based on the cutoff functions \eqref{def:cutoff.beta} and \eqref{def:cutoff.beta2}, can easily be adapted to obtain the required compactness results. In this case, the convergence $\Pi_{\mD_l}u_l\to \bar u$ strongly in $L^\infty(0,T;L^{m+1}(\Omega))$ has to be replaced with the convergence $\int_\Omega \zeta(\Pi_{\mD_l}u_l(\cdot,\x))\d\x\to \int_\Omega\zeta(u(\cdot,\x))d\x$ uniformly on $[0,T]$ (where we recall that $\zeta(z)=\int_0^z \beta(s)ds$). In case of non-linearities $\beta$ with plateaux, though, the dependency of $\Lambda$ has to be through $\beta(u)$; \cite{Droniou2016} for details.
\end{remark}


\section{A priori estimates}\label{sec:apriori}

We first provide a priori estimates for the solution $u$ to~\eqref{eq: gdmscheme}, and then deduce its existence in Corollary~\ref{co: existgdm}. For legibility, we drop the index $l$ in sequences of gradient discretisations, and we simply write $\disc$ instead of $\disc_l$.

\begin{lemma}\label{lem: priori}
Let $\zeta$ be defined by \eqref{def:zeta}, and $u$ be a solution of~\eqref{eq: gdmscheme}. Then, for any integer number $k\in [1,N]$, we have
\begin{equation}\label{eq: priori1}
\begin{aligned}
\int_\Omega \zeta(\Pi_\mD u^{(k)})(\x)d\x
+ &\big\langle \Lambda(\cdot,\Pi_\mD u)\nabla_\mD\beta(u),
\nabla_\mD\beta(u)\big\rangle_{\Omega_{t^{(k)}}}\\
&\leq \int_\Omega\zeta(\Pi_\mD u^{(0)})(\x)d\x +\big\langle f,  \Pi_\mD \beta(u)\big\rangle_{L^2(\Omega_{t^{(k)}})}.
\end{aligned}
\end{equation}
Consequently, there exists a constant $C>0$ depending only on $\underline{\lambda}$, $f$, $C_p\geq C_{\mD}$  and $C_{\mathrm{ini}}\ge\|\Pi_\mD \mI_\mD u_0\|_{L^{m+1}(\Omega)}$  such that 
\begin{align}\label{eq: priori2}
\|\Pi_\mD u\|_{L^\infty(0,T;L^{m+1}(\Omega))}
+
\|\Pi_\mD \zeta( u)\|_{L^\infty(0,T;L^1(\Omega))}
+ \|\nabla_\mD\beta(u)\|_{L^2(\Omega_T)}
\leq 
C.
\end{align}
\end{lemma}
\begin{proof}
We choose the test function $\phi = \bigl(\beta(u^{(0)}),\cdots,\beta(u^{(k)}),0,\cdots,0\bigr)\subset X_{\mD,0}$ in~\eqref{eq: gdmscheme} to write
\begin{equation}\label{eq: pri1}
\begin{aligned}
\big\langle \delta_\mD u,\Pi_\mD \beta(u)\big\rangle_{L^2(\Omega_{t^{(k)}})}
&{}+
\big\langle \Lambda(\cdot,\Pi_\mD u)\nabla_\mD\beta(u),
\nabla_\mD\beta(u)\big\rangle_{\Omega_{t^{(k)}}}\\
={}&
\big\langle f,  \Pi_\mD \beta(u)\big\rangle_{L^2(\Omega_{t^{(k)}})}.
\end{aligned}
\end{equation}
For any $n = 0,\cdots,N-1$ and $t\in (t^{(n)},t^{(n+1)}]$, we estimate the first term in the left hand side of~\eqref{eq: pri1}, starting from
\begin{align}\label{eq: pri2}
\delta_\mD u (t)\,\Pi_\mD \beta(u^{(n+1)})
=
\frac{1}{\delta t^{(n+\frac12)}}
\big(\Pi_\mD u^{(n+1)} - \Pi_\mD u^{(n)}
\big)\beta(\Pi_\mD u^{(n+1)}).
\end{align}
Since $\beta$ is increasing,
 $\zeta$ is convex and thus above its tangent line, which implies $\zeta(b)-\zeta(a)\leq (b-a)\beta(b)$ for all $a,b\in\R$.
Applying this inequality with $a =\Pi_\mD u^{(n)}$ and $b =\Pi_\mD u^{(n+1)}$, it follows from~\eqref{eq: pri2} that
\begin{equation}\label{eq: pri3}
\frac{1}{\delta t^{(n+\frac12)} }
\big(\zeta(\Pi_\mD u^{(n+1)})-\zeta(\Pi_\mD u^{(n)}) \big)
\leq \delta_\mD u (t)\,\Pi_\mD \beta(u^{(n+1)}).
\end{equation}
Plugging that in~\eqref{eq: pri1} and using a telescopic sum yields \eqref{eq: priori1}.
The \emph{a priori} estimates \eqref{eq: priori2} follow from this relation and \eqref{eq:hyp.Lambda}, by writing
\begin{align*}
\int_\Omega \zeta(\Pi_\mD u^{(k)})(\x)d\x
{}&+ \underline{\lambda}\|\nabla_\mD\beta(u)\|_{L^2(\Omega_{t^{(k)}})}^2\\
\leq{}&\int_\Omega \zeta(\Pi_\mD u^{(k)})(\x)d\x
+ \big\langle \Lambda(\cdot,\Pi_\mD u)\nabla_\mD\beta(u),
\nabla_\mD\beta(u)\big\rangle_{\Omega_{t^{(k)}}}\\
\leq{}&
\int_\Omega\zeta(\Pi_\mD u^{(0)})(\x)d\x
+
\big\langle f,  \Pi_\mD \beta(u)\big\rangle_{L^2(\Omega_{t^{(k)}})}\\
\leq{}&
\int_\Omega\zeta(\Pi_\mD u^{(0)})(\x)d\x
+\frac{C_p^2}{2\underline{\lambda}}\|f\|^2_{L^2(\Omega_{t^{(k)}})}+\frac{\underline{\lambda}}{2C_p^2}\|\Pi_\mD \beta(u)\|^2_{L^2(\Omega_{t^{(k)}})}.
\end{align*}
Noting that $\zeta(a) = \frac{|a|^{m+1}}{(m+1)}\ge 0$ for all $a\in \R$ and recalling that
\begin{equation}\label{eq: coer}
\|\Pi_\mD \beta(u)\|_{L^2(\Omega_{t^{(k)}})}
\leq C_\mD \|\nabla_\mD \beta(u)\|_{L^2(\Omega_{t^{(k)}})}\le C_p \|\nabla_\mD \beta(u)\|_{L^2(\Omega_{t^{(k)}})},
\end{equation}
we obtain the estimates in \eqref{eq: priori2}.
\end{proof}
The following corollary guarantees the existence of a solution to~\eqref{eq: gdmscheme}.
\begin{corollary}\label{co: existgdm}
 If $\mD$ is a gradient discretisation in the sense of Definition~\ref{def: gdm} 
then there exists at least one solution to the gradient scheme~\eqref{eq: gdmscheme}.
\end{corollary}
\begin{proof}
The result is obtained using~\eqref{eq: priori2}, the properties $|\beta(u)|= |u|^m$ and $u\beta(u)= |u|^{m+1}$,
and the same arguments as in~\cite[Corrollary 4.2]{Droniou2016} 
\end{proof}

In the following lemma, we estimate the dual norm \eqref{def:dualnorm} of the discrete time derivative $\delta_\mD u$. 
\begin{lemma}\label{lem: dualnorm}
There exists a constant $C$  depending on $\overline{\lambda}$, $\underline{\lambda}$, $f$, $C_p\geq C_{\mD}$ and $C_{\mathrm{ini}}\ge\|\Pi_\mD \mI_\mD u_0\|_{L^{m+1}(\Omega)}$ such that
\[
\int_0^T|\delta_\mD u(t)|^2_{*,\mD}dt\leq C.
\]
\end{lemma}
\begin{proof}
We first fix $k\in\{0,\cdots,N-1\}$ and then, for any $\xi\in X_{\mD,0}$, take $\phi=(\phi^{(n)})_{n=0,\ldots,N}\subset X_{\mD,0}$ in~\eqref{eq: gdmscheme} such that $\phi^{(k+1)} =\xi$ and $\phi^i=0$ for all $i\not=k+1$.
Using the Cauchy--Schwarz inequality and the definition of $C_\mD$ we deduce 
\begin{align*}
\delta t^{(k+\frac12)} \big\langle \delta_\mD^{(k+\frac12)} u,\Pi_\mD \xi\big\rangle_{L^2(\Omega)}
&=
-
\big\langle\Lambda(\cdot,\Pi_\mD u)\nabla_\mD\beta(u),  \nabla_\mD \phi\big\rangle_{L^2(\Omega_T)}
+
\big\langle f,  \Pi_\mD \phi\big\rangle_{L^2(\Omega_T)}\\
&\leq \overline{\lambda}\delta t^{(k+\frac12)}\|\nabla_\mD\beta(u^{(k+1)})\|_{L^2(\Omega)}\|\nabla_\mD \xi\|_{L^2(\Omega)}\\
&\quad+\delta t^{(k+\frac12)}
\|f^{(k+1)}\|_{L^2(\Omega)}C_\mD\|\nabla_\mD \xi\|_{L^2(\Omega)},
\end{align*}
where $f^{(k+1)}(\x)=\frac{1}{\delta t^{(k+\frac12)}}\int_{t^{(k)}}^{t^{(k+1)}} f(t,\x)dt$. Taking the supremum over all $\xi\in X_{\mD,0}$ satisfying $\|\nabla_\mD \xi\|_{L^2(\Omega)}=1$, we deduce
\[
\delta t^{(k+\frac12)} |\delta_\mD^{(k+\frac12)} u|_{*,\mD}\leq\overline{\lambda} \delta t^{(k+\frac12)}\|\nabla_\mD\beta(u^{(k+1)})\|_{L^2(\Omega)}+\delta t^{(k+\frac12)}
\|f^{(k+1)}\|_{L^2(\Omega)}C_\mD.
\]
Divide by $(\delta t^{(k+\frac12)})^{1/2}$, square, and sum over $k$ (and use Jensen or Cauchy--Schwarz to estimate the term involving $f^{(k+1)}$) to get
\[
\int_0^T |\delta_\mD u(t)|_{*,\mD}^2dt\le 2\overline{\lambda}^2\|\nabla_\mD\beta(u)\|_{L^2(\Omega_T)}^2 + 2C_\mD^2 \|f\|_{L^2(\Omega_T)}^2.
\]
This together with~\eqref{eq: priori2}  complete the proof of the lemma.
\end{proof}

To deal with the lack of global Lipschitz estimates on $\beta$, we introduce cutoff functions. The definition of these functions is different in the case $0<m<1$ (fast diffusion) and $m>1$ (slow diffusion), because each of these cases correspond to a certain breakdown in the Lipschitz properties of $\beta$: for fast diffusion, $\beta$ is not Lipschitz-continuous at 0, whereas for slow diffusion, the Lipschitz constant of $\beta$ explodes at $\pm\infty$.
We therefore define $\beta_k^\s$ and $\beta_k^\f$ for any $k>0$ as
\begin{equation}\label{def:cutoff.beta}
\text{for }0<m<1,\quad
\beta_k^\f(r):=
\begin{cases}
k^{1-m}\,r \quad &\text{for } - 1/k\leq  r \le 1/k\\
\beta(r)\quad &\text{for } r<-1/k\mbox{ or }  r> 1/k,
\end{cases}
\end{equation}
and
\begin{equation}\label{def:cutoff.beta2}
\text{for } m>1,\quad
\beta_k^\s(r):=
\begin{cases}
k^m\quad &\text{for } r\ge k\\
\beta(r)\quad &\text{for } -k<r < k\\
-k^m\quad &\text{for } r\le -k.
\end{cases}
\end{equation}
We notice that, in the case $m>1$, such a regularisation of $\beta$ has already been considered in the literature, see e.g. \cite{NochettoVerdi88,Ben07,Pop02}.
The cutoff functions $\beta_k^\s$ and $\beta_k^\f$ are globally Lipschitz continuous with respective Lipschitz constants $L_k^\f = k^{1-m}$ and $L_k^\s = m\,k^{m-1}$. 
Our goal here is to estimate the time-translates of $\beta(u)$. To achieve this, we will be using the cutoff functions, as well as the following relation.

\begin{lemma}\label{lem: betak} 
Recalling that $\ms=\max(1,1/m)$, for any  $ a, b \in \R$ we have
\[
\big(\beta_k^\ii (b) -\beta_k^\ii (a)\big)^2
\leq \ms L_k^\ii (b-a )\big(\beta (b) -\beta (a)\big),\quad\text{for } \ii=\s,\f.
\]
\end{lemma}
\begin{proof}
By noting that $\beta'(r) = m|r|^{m-1}$ for $r\in\R\backslash\{0\}$ and
\[
\big(\beta_k^\f\big)'(r):=
\begin{cases}
k^{1-m} \quad &\text{for } - 1/k<  r < 1/k\\
\beta'(r)\quad &\text{for } r<-1/k\mbox{ or }  r> 1/k,
\end{cases}
\]
\[
\big(\beta_k^\s\big)'(r):=
\begin{cases}
0 \quad &\text{for }  r<-k\mbox{ or } r> k,\\
\beta'(r)\quad &\text{for } -k<r < k,
\end{cases}
\]
we obtain 
\begin{align*}
0\leq \big(\beta_k^\f\big)'(r)&\leq \frac{1}{m}\beta'(r)\qquad\forall r\in\R\backslash\{-1/k,0,1/k\}\\
0\leq \big(\beta_k^\s\big)'(r)&\leq \beta'(r)\qquad\forall r\in\R\backslash\{-k,0,k\}.
\end{align*}
The above inequalities imply that for any $ a, b \in \R$,
\begin{equation*}
\left|\int_a^b \big(\beta_k^\ii\big)' (s) ds\right|\leq \ms\left|\int_a^b \beta' (s) ds\right|,\quad
\text{for }\ii=\s,\f ,
\end{equation*}
or, equivalently,
\begin{equation}
\big|\beta_k^\ii (b) -\beta_k^\ii (a)\big|
\leq\ms
\big|\beta (b)- \beta (a)\big|,\quad
\text{for } \ii=\s,\f.
\end{equation}
Together with the global Lipschitz property of $\beta_k^\ii$ and the monotonicity of $\beta$, this yields
\begin{align*}
\big(\beta_k^\ii (b) -\beta_k^\ii (a)\big)^2
&\leq 
 L_k^\ii |b-a | \big|\beta_k^\ii (b) -\beta_k^\ii (a)\big|\\
&\leq
 L_k^\ii |b-a | \ms\big|\beta (b)- \beta (a)\big|
 =\ms
 L_k^\ii (b-a ) \big(\beta (b)- \beta (a)\big),
\end{align*}
which completes the proof.
\end{proof}

We can now estimate the time translates of $\beta(u)$. 

\begin{lemma}\label{lem: timetrans}
Let $\mD$ be a gradient discretisation in the sense of Definition~\ref{def: gdm} 
and $u$ be  a solution to scheme~\eqref{eq: gdmscheme}. Then,
there exists a constant $C$  depending only on $m$, $\overline{\lambda}$, $\underline{\lambda}$, $f$, $C_p\ge C_{\mD}$ and $C_{\mathrm{ini}}\ge\|\Pi_\mD \mI_\mD u_0\|_{L^{m+1}(\Omega)}$ such that:
\begin{itemize}
\item for $m\in(0,1)$,
\begin{equation}\label{eq: time1}
\big\|\Pi_\mD\beta(u)(\cdot+\tau,\cdot)-\Pi_\mD\beta(u)\big\|^2_{L^2(\Omega_{T-\tau})}
\leq C(\delta t_\mD + \tau)^{2m/(m+1)},
\end{equation}
\item for $m>1$,
\begin{equation}\label{eq: time2}
\big\|\Pi_\mD\beta(u)(\cdot+\tau,\cdot)-\Pi_\mD\beta(u)\big\|^2_{L^1(\Omega_{T-\tau})}
\leq 
C(\delta t_\mD + \tau)^{2/(m+1)}.
\end{equation}
\end{itemize}
\end{lemma}
\begin{proof}
In this proof, $C$ denotes a generic constant that has the same dependencies as in the lemma.
For $\ii=\s,\f$ and any $k\in\N$ we have
\begin{align}\label{eq: timetrans1}
\Pi_\mD\beta(u)&(\cdot+\tau,\cdot)-\Pi_\mD\beta(u)
=
\Gamma_{1,k}^\ii + \Gamma_{2,k}^\ii
+
\Gamma_{3,k}^\ii
\end{align}
with 
\begin{align*}
\Gamma_{1,k}^\ii
:= \Pi_\mD\big(\beta_k^\ii(u)-\beta(u)\big),\quad
\Gamma_{2,k}^\ii
:=\Pi_\mD\big(\beta(u)-\beta_k^i(u)\big)(\cdot+\tau,\cdot),
\end{align*}
and
\[
\Gamma_{3,k}^\ii:=
\Pi_\mD\beta_k^\ii(u)(\cdot+\tau,\cdot)-\Pi_\mD\beta_k^\ii(u).
\]
To estimate $\Gamma_{1,k}^\ii$, define
\begin{align*}
\Omega_k^\f:={}&\{(t,\x)\in \Omega_{T-\tau}\,:\, |\Pi_\mD u(t,\x)|<1/k\},\\
\Omega_k^\s:={}&\{(t,\x)\in \Omega_{T-\tau}\,:\, |\Pi_\mD u(t,\x)|> k\}.
\end{align*}
The commutativity property \eqref{PiD:commute} shows that $\big|\Pi_\mD\big(|u|^{m-1}u\big)\big|=
\Pi_\mD|u|^{2m}=|\Pi_\mD u|^{2m}$ and thus, when $\ii=\f$,
\begin{align}\label{eq: timetrans21}
\big\|\Gamma_{1,k}^\f \big\|^2_{L^2(\Omega_{T-\tau})}&= \int_{\Omega_{T-\tau}} \mId_{\Omega_k^\f}\big|\Pi_\mD\big(|u|^{m-1}u-\beta_k^\f(u)\big)\big|^2(t,\x) d\x dt\nonumber\\
&= \int_{\Omega_k^\f} \big|\Pi_\mD\big(|u|^{m-1}u-k^{1-m}u\big)\big|^2(t,\x) d\x dt\nonumber\\
&\leq
2\int_{\Omega_k^\f} |\Pi_\mD u|^{2m}(t,\x) + k^{2-2m}|\Pi_\mD u|^2(t,\x) d\x dt\nonumber\\
&\leq
4k^{-2m} \big|\Omega_k^\f\big|\leq  4 \big|\Omega_T\big|k^{-2m}.
\end{align}
By the H\"older and Chebyshev inequalities,~\eqref{PiD:commute} and~\eqref{eq: priori2}, we estimate $\Gamma_{1,k}^\s$:
\begin{align}\label{eq: timetrans3}
\big\|\Gamma_{1,k}^\s \big\|_{L^1(\Omega_{T-\tau})}
&= \int_{\Omega_{T-\tau}} \mId_{\Omega_k^\s}\big|\Pi_\mD\big(|u|^{m-1}u-\beta_k^\s(u)\big)\big|(t,\x) d\x dt\nonumber\\
&\leq
\int_{\Omega_{T-\tau}} \mId_{\Omega_k^\s} \Pi_\mD\big(|u|^{m} + k^{m}\big)(t,\x) d\x dt\nonumber\\
&\leq
2\int_{\Omega_{T-\tau}} \mId_{\Omega_k^\s} |\Pi_\mD u|^{m} (t,\x) d\x dt\nonumber\\
&\leq
2 \|\Pi_\mD u\|_{L^{m+1}(\Omega_{T-\tau})}^{m}\big|\Omega_k^\s\big|^{1/(m+1)}\nonumber\\
&\leq
2k^{-1} \|\Pi_\mD u\|_{L^{m+1}(\Omega_{T-\tau})}^{m+1}\leq Ck^{-1}.
\end{align}

Similarly, we also have estimates for $\Gamma_{2,k}^\ii$:
\begin{equation}\label{eq: timetrans4}
\big\|\Gamma_{2,k}^\f \big\|^2_{L^2(\Omega_{T-\tau})}
\leq  Ck^{-2m},\quad
\big\|\Gamma_{2,k}^\s \big\|_{L^1(\Omega_{T-\tau})}
\leq Ck^{-1}.
\end{equation}

We estimate $\Gamma_{3,k}^\ii$ using the same arguments in~\cite[Lemma 4.4]{Droniou2016}. Thanks to Lemma~\ref{lem: betak} we have
\begin{align}\label{eq: timetrans5}
\big\|\Gamma_{3,k}^\ii \big\|^2_{L^2(\Omega_{T-\tau})}
\leq
 \ms L_k^\ii \int_0^{T-\tau} A(t,\tau)dt,
\end{align}   
with
\[
A(t,\tau):=\int_\Omega
 \big(\Pi_\mD u(t+\tau,\x)- \Pi_\mD u(t,\x)\big)\big(\Pi_\mD \beta(u)(t+\tau,\x)- \Pi_\mD \beta(u)(t,\x)\big)d\x. 
\]
For any $t\in (0,T)$ there exists $n_t\in \{0,\cdots,N-1\}$ such that $t^{(n_t)}<t\leq t^{(n_t + 1)}$. We rewrite $A$ by expressing $\Pi_\mD u(t+\tau,\x)- \Pi_\mD u(t,\x)$ as the sum of its jumps in time, and use the definition of dual semi-norm $|\cdot|_{*,\mD}$ to infer
\begin{align*}
A(t,\tau)
&=\sum_{j=n_t+1}^{n_{t+\tau}}\delta t^{(j+\frac12)}\int_\Omega
   \delta_\mD^{(j+\frac12)}u(\x)\,\Pi_\mD \big(\beta(u^{(n_{t+\tau} +1)})- \beta(u^{(n_{t} +1)})\big)(\x)d\x\\
  &\leq 
    \|\nabla_\mD \big(\beta(u^{(n_{t+\tau} +1)})- \beta(u^{(n_{t} +1)})\big)\|_{L^2(\Omega)} \sum_{j=n_t+1}^{n_{t+\tau}}\delta t^{(j+\frac12)}
|\delta_\mD^{(j+\frac12)}u|_{*,\mD}\\
&=
\|\nabla_\mD \beta (u)(t+\tau)- \nabla_\mD\beta (u)(t)\|_{L^2(\Omega)} \sum_{j=n_t+1}^{n_{t+\tau}}\delta t^{(j+\frac12)}
|\delta_\mD^{(j+\frac12)}u|_{*,\mD}.
\end{align*}
Together with the Cauchy--Schwarz inequality (on the sums and the integrals) and~\eqref{eq: priori2}, this implies
\begin{align}\label{eq: timetrans6}
\int_0^{T-\tau} &A(t,\tau)dt
\leq 
\bigg[\int_0^{T-\tau}\|\nabla_\mD \beta (u)(t+\tau)- \nabla_\mD\beta (u)(t)\|_{L^2(\Omega)}^2 dt \bigg]^{1/2}\nonumber\\
&\qquad\qquad\qquad \times\bigg[\int_0^{T-\tau}\big(\sum_{j=n_t+1}^{n_{t+\tau}}\delta t^{(j+\frac12)}
|\delta_\mD^{(j+\frac12)}u|_{*,\mD}\big)^2dt \bigg]^{1/2}\nonumber\\
&\leq 
C\bigg[\int_0^{T-\tau}\big(\sum_{j=n_t+1}^{n_{t+\tau}}\delta t^{(j+\frac12)}\big)
\big(\sum_{j=n_t+1}^{n_{t+\tau}}\delta t^{(j+\frac12)}
|\delta_\mD^{(j+\frac12)}u|_{*,\mD}^2\big)dt \bigg]^{1/2}.
\end{align}
For any $t\in (0,T-\tau)$, we note that
\begin{align*}
\sum_{j=n_t+1}^{n_{t+\tau}}\delta t^{(j+\frac12)} 
&= 
t^{(n_{t+\tau}+1)} - t^{(n_{t}+1)}\nonumber\\
&=
\big[t^{(n_{t+\tau}+1)} - (t+\tau)\big] + \big(t-t^{(n_{t}+1)}\big) + \tau
\leq 
2\delta t_\mD +\tau. 
\end{align*}
Hence, \eqref{eq: timetrans6} and Lemma \ref{lem: dualnorm} yield
\begin{align*}
\int_0^{T-\tau} A(t,\tau)dt
&\leq 
C(2\delta t_\mD + \tau)^{1/2}\bigg[\int_0^{T-\tau}\int_{t^{(n_t+1)}}^{t^{(n_{t+\tau})}}|\delta_\mD u(s)|_{*,\mD}^2 ds\,dt \bigg]^{1/2}\\
&\leq 
C(2\delta t_\mD + \tau)^{1/2}\bigg[\int_0^{T}|\delta_\mD u(s)|_{*,\mD}^2\int_{s-\tau-\delta t_\mD}^s dt\,ds \bigg]^{1/2}\\
&\leq 
C(\delta t_\mD + \tau).
\end{align*}
Together with~\eqref{eq: timetrans5}, this implies
\begin{align}\label{eq: timetrans7}
\big\|\Gamma_{3,k}^\ii \big\|^2_{L^2(\Omega_{T-\tau})}
\leq
 C L_k^\ii (\delta t_\mD + \tau),\quad\text{for } \ii=\s,\f.
\end{align} 
From~\eqref{eq: timetrans1}--\eqref{eq: timetrans4} and~\eqref{eq: timetrans7}, we deduce that 
\begin{align*}
\mbox{for $m<1$:}&\quad\big\|\Pi_\mD\beta(u)(\cdot+\tau,\cdot)-\Pi_\mD\beta(u)\big\|^2_{L^2(\Omega_{T-\tau})}
\leq Ck^{-2m} + Ck^{1-m}(\delta t_\mD + \tau),\\
\mbox{for $m>1$:}&\quad\big\|\Pi_\mD\beta(u)(\cdot+\tau,\cdot)-\Pi_\mD\beta(u)\big\|^2_{L^1(\Omega_{T-\tau})}
\leq 
Ck^{-2} + Cmk^{m-1}(\delta t_\mD + \tau).
\end{align*}
The results~\eqref{eq: time1} and~\eqref{eq: time2} follow from the above inequalities by choosing 
$k = (\delta t_\mD + \tau )^{-1/(m+1)}$.
\end{proof}
\section{Initial convergence of gradient schemes}\label{sec:init_conv}

Let us start with a strong convergence result on $\Pi_{\mD_l}\beta(u_l)$.

\begin{lemma}\label{lem: strongconv}
Let $r=2$ if $m\in (0,1)$ and $r=1$ if $m>1$.
There exists a subsequence of $\big(\Pi_{\mD_l}\beta(u_l)\big)_{l\ge 1}$ (still denoted by $\big(\Pi_{\mD_l}\beta(u_l)\big)_{l\ge 1}$) and $\bar{\beta} \in L^r(\Omega_T)$ such that, as $l\to\infty$,
\[
\Pi_{\mD_l}\beta(u_l) \rightarrow \bar{\beta}\quad\text{strongly in } L^r(\Omega_T).
\]
\end{lemma}
\begin{proof}
The result is obtained using~\eqref{eq: priori2} and the compactness of $(\mD_l)_{l\ge 1}$ (to estimate the space translates of $\Pi_{\mD_l}\beta(u_l)$), Lemma~\ref{lem: timetrans} (to estimate the time translates of $\Pi_{\mD_l}\beta(u_l)$), Kolmogorov's theorem and the same arguments as in~\cite[p.748]{Droniou2016}
\end{proof}
\begin{lemma}\label{lem: uniform weak conv}
Let $T>0$ and take a sequence $(\mD_l)_{l\ge 1}$ of space-time gradient discretisations, in the sense of Definition~\ref{def: gdm}, that is consistent. Let $u_l$ be a solution to \eqref{eq: gdmscheme} with $\mD=\mD_l$.
Then, the sequence $(\Pi_{\mD_l}u_l)_{l\ge 1}$ is relatively compact uniformly-in-time and weakly in $L^{m+1}(\Omega)$, i.e. 
there exists a subsequence of $(\Pi_{\mD_l}u_l)_{l\ge 1}$ (still denoted by $(\Pi_{\mD_l}u_l)_{l\ge 1}$) and a function $\bar{u}:[0,T]\rightarrow L^{m+1}(\Omega)$ such that, 
for all $\phi\in L^{1+1/m}(\Omega)$, the sequence of functions 
\[
t\in [0,T]\mapsto \, _{L^{m+1}}\langle \Pi_{\mD_l}u_l(t),\phi \rangle_{L^{1+1/m}}
\]
 converges uniformly on $[0,T]$ to the function 
\[
t\in [0,T]\mapsto \, _{L^{m+1}}\langle \bar{u}(t),\phi \rangle_{L^{1+1/m}}.
\]
Moreover, $\bar{u}$ is continuous $[0,T]\mapsto L^{m+1}(\Omega)$ for the weak topology of $L^{m+1}(\Omega)$.
\end{lemma}
\begin{proof}
The result is a consequence of the discontinuous Ascoli--Arzela theorem~\cite[Theorem 6.2]{Droniou2016} (see also \cite[Theorem C.11, p455]{Droniou.et.al2018}). Let us check the assumptions of this theorem.

Let $(\phi_i)_{i\in\N}\subset C_c^\infty(\Omega)$ be a dense sequence in $L^{1+1/m}(\Omega)$ and equip the ball $B$ of radius $C$ (from \eqref{eq: priori2}) in $L^{m+1}(\Omega)$ with the following metric
\[
d_B(v,w) = \sum_{i\in\N} \frac{\text{min}(1,| _{L^{m+1}}\langle v-w,\phi_i\rangle_{L^{1+1/m}}|)}{2^i}\quad \text{for }v,w\in B.
\]
The metric $d_B$ defines the weak topology of $L^{m+1}(\Omega)$ on $B$, and
the set $B$ is metric compact and therefore complete for this weak topology. It follows from~\eqref{eq: priori2} that $\Pi_{\mD_l}u_l (t,\cdot)\in B$ for $t\in [0,T]$. It remains to estimate $d_B\bigl(\Pi_{\mD_l}u_l(s),\Pi_{\mD_l}u_l(s')\bigr)$ for $0\leq s\leq s'\leq T$. In the following, $C$ denotes a generic constant that may change from one line to the next but does not depend on $l$ or $i$.

We first define the interpolator $P_{\mD_l} : H^1_0(\Omega) \rightarrow X_{\mD_l,0}$ by 
\begin{equation}\label{def P}
P_{\mD_l} \phi := \text{argmin}_{w\in X_{\mD_l,0}} 
\bigl(\|\Pi_{\mD_l} w - \phi\|_{L^{1+\ms}(\Omega)} + \|\nabla_{\mD_l} w-\nabla \phi\|_{L^2(\Omega)}\bigr).
\end{equation}
We rewrite $\Pi_{\mD_l}u_l(s')-\Pi_{\mD_l}u_l(s)$ as the sum of its jumps $\delta t^{(j+\frac12)} \delta_{\mD_l}^{(j+\frac12)} u_l$ at points $\bigl(t^{(n)}\bigr)_{n=n_1,\cdots,n_2}$ between $s$ and $s'$. Using the definition of $|{\cdot}|_{*,\mD_l}$, the Cauchy--Schwarz inequality and Lemma~\ref{lem: dualnorm} we obtain
\begin{align}\label{eq: disAS1}
\bigg|\int_\Omega \bigg(\Pi_{\mD_l}u_l(s',\x)-\Pi_{\mD_l}&u_l(s,\x)\bigg)\Pi_{\mD_l}P_{\mD_l} \phi_i(\x) d\x\bigg|\nonumber\\
&=
\bigg|\int_{t^{(n_1)}}^{t^{(n_2+1)}}\int_\Omega \delta_{\mD_l} u(t,\x)\Pi_{\mD_l}P_{\mD_l} \phi_i(\x) d\x dt\bigg|\nonumber\\
&\leq 
\int_{t^{(n_1)}}^{t^{(n_2+1)}} |\delta_{\mD_l} u_l(t)|_{*,\mD_l}\|\nabla_{\mD_l} P_{\mD_l} \phi_i\|_{L^2(\Omega)} dt\nonumber\\
&\leq 
C^{1/2} (t^{(n_2+1)}-t^{(n_1)})^{1/2}\|\nabla_{\mD_l} P_{\mD_l} \phi_i\|_{L^2(\Omega)}.
\end{align}
By noting that $t^{(n_2+1)}-t^{(n_1)}\leq |s'-s| + \delta t_{\mD_l}$, we deduce from~\eqref{eq: priori2} and~\eqref{eq: disAS1} that 
\begin{align}\label{eq: disAS2}
\bigg|\int_\Omega &\bigg(\Pi_{\mD_l}u_l(s',\x)-\Pi_{\mD_l}u_l(s,\x)\bigg)\phi_i(\x) d\x\bigg|\nonumber\\
&\leq
\bigg|\int_\Omega \bigg(\Pi_{\mD_l}u_l(s',\x)-\Pi_{\mD_l}u_l(s,\x)\bigg)\Pi_{\mD_l}P_{\mD_l} \phi_i(\x) d\x\bigg|\nonumber\\
&\quad +
\bigg|\int_\Omega \bigg(\Pi_{\mD_l}u_l(s',\x)-\Pi_{\mD_l}u_l(s,\x)\bigg)\bigg(\Pi_{\mD_l}P_{\mD_l} \phi_i(\x)- \phi_i(\x)\bigg) d\x\bigg|\nonumber\\
&\leq 
C \big( |s'-s| + \delta t_{\mD_l}\big)^{1/2}\|\nabla_{\mD_l} P_{\mD_l} \phi_i\|_{L^2(\Omega)}
+
C\|\Pi_{\mD_l}P_{\mD_l} \phi_i- \phi_i\|_{L^{1+1/m}(\Omega)}.
\end{align}
It follows from~\eqref{def P} and the consistency of $(\mD_l)_{l\ge 1}$ that $\|\Pi_{\mD_l}P_{\mD_l}\phi_i - \phi_i\|_{L^{1+1/m}(\Omega)}\leq C\hat{S}_{\mD_l}(\phi_i)$ and $\|\nabla_{\mD_l} P_{\mD_l} \phi_i\|_{L^2(\Omega)}
\leq 
\hat{S}_{\mD_l}(\phi_i) + \|\nabla \phi_i\|_{L^2(\Omega)}$
Since $\hat{S}_{\mD_l}(\phi_i) \rightarrow 0 $ as $l\rightarrow \infty$, there exists a constant $C_{\phi_i}$  depending only on $\phi_i$ such that $\hat{S}_{\mD_l}(\phi_i) + \|\nabla \phi_i\|_{L^2(\Omega)} \leq C_{\phi_i}$. Hence, we estimate the right hand side of~\eqref{eq: disAS2} to obtain
\begin{equation*}
\bigg|\int_\Omega \bigg(\Pi_{\mD_l}u_l(s',\x)-\Pi_{\mD_l}u_l(s,\x)\bigg)\phi_i(\x) d\x\bigg|
\leq 
C \big( |s'-s| + \delta t_{\mD_l}\big)^{1/2} C_{\phi_i}
+
C\hat{S}_{\mD_l}(\phi_i).
\end{equation*}
Together with the definition of metric $d_B$, this implies
\begin{align*}
d_B\bigl(\Pi_{\mD_l}u_l(s),\Pi_{\mD_l}u_l(s')\bigr) 
&\leq 
\sum_{i\in\N}  \frac{\text{min}(1,C |s'-s|^{1/2} C_{\phi_i})}{2^i}\\
&\quad+
\sum_{i\in\N}  \frac{\text{min}(1,C\delta t_{\mD_l}^{1/2} C_{\phi_i} + C\hat{S}_{\mD_l}(\phi_i))}{2^i}.
\end{align*}
Using the dominated convergence theorem for series and the fact that, for any $i\in\N$, $\hat{S}_{\mD_l}(\phi_i) \rightarrow 0 $ and $\delta t_{\mD_l}\rightarrow 0$ as $l\rightarrow \infty$, we see that 
\begin{align*}
&\lim_{|s'-s|\rightarrow 0} \sum_{i\in\N}  \frac{\text{min}(1,C |s'-s|^{1/2} C_{\phi_i})}{2^i} = 0\\
\text{and}\quad
&\lim_{l\rightarrow \infty} \sum_{i\in\N}  \frac{\text{min}(1,C\delta t_{\mD_l}^{1/2} C_{\phi_i} + C\hat{S}_{\mD_l}(\phi_i))}{2^i} = 0.
\end{align*}
Hence the assumptions of the discontinuous Ascoli--Arzela theorem are satisfied and the proof is complete.
\end{proof}

We recall the following lemma (see e.g.~\cite[Lemma 4.8]{Droniou.et.al2018}).

\begin{lemma}[Regularity of the limit]\label{lem: regu lim}
Let  $(\mD_l)_{l\ge 1}$ be a sequence of space-time gradient discretisations, in the sense of Definition~\ref{def: gdm}, that is limit-conforming. Let $v_l:=(v_l^{(n)})_{n=0,\cdots,N_l}\subset X_{\mD_l}$ for any $l\in \N$, be such that 
$\bigl(\nabla_{\mD_l}v_l\bigr)_{l\in \N}$ is bounded in $L^2(\Omega_T)^d$.

Then, there exists $v\in L^2(0,T;H^1_0(\Omega))$ such that, along a subsequence as $l\to\infty$,
\[
\Pi_{\mD_l}v_l\rightarrow v\text{ weakly in } L^2(\Omega_T)\mbox{ and }
\nabla_{\mD_l}v_l\rightarrow v\text{ weakly in } L^2(\Omega_T)^d.
\]
\end{lemma}

We can now prove a preliminary and weaker version of Theorem \ref{theo:cv.GS}, in which the convergences are
weak.

\begin{theorem}\label{theo:cv.GS:weak}
Under the assumptions of Theorem \ref{theo:cv.GS}, there exists a weak solution $\bar u$ to \eqref{eq: PME} in the sense of Definition~\ref{def: weak sol} such that, up to a subsequence as $l\to\infty$,
\begin{itemize}
\item $\Pi_{\mD_l}\beta(u_l)\to \beta(u)$ weakly in $L^2(\Omega_T)$,
\item $\nabla_{\mD_l}\beta(u_l)\to \nabla \beta(u)$ weakly in $L^2(\Omega_T)^d$.
\end{itemize}
\end{theorem}

\begin{proof}

\underline{Step 1: Convergence of discrete solutions}.

Using the Minty trick~\cite[Lemma 3.5]{Droniou2016} recalled in Lemma~\ref{lem: ap1}, we deduce from Lemmas~\ref{lem: uniform weak conv} and~\ref{lem: strongconv} that $\bar{\beta} = \beta(\bar{u})$ a.e.\ on $\Omega_T$, that $\bar u\in C([0,T];L^{m+1}(\Omega)_{\mathrm{w}})$, and that $\Pi_{\mD_l}u_l\to \bar u$ uniformly-in-time and weakly in $L^{m+1}(\Omega)$. Moreover, by consistency of the gradient discretisations, $\Pi_{\mD_l}u_l(0)= \Pi_{\mD_l} \mI_{\mD_l}u_0\to u_0$ in $L^{m+1}(\Omega)$. Hence, the uniform-in-time weak-in-space convergence shows that $\bar u(0)=u_0$.
Owing to Lemma~\ref{lem: strongconv}, the estimate~\eqref{eq: priori2} and  Lemma~\ref{lem: regu lim} we have 
$\beta(\bar{u})\in L^2(0,T;H^1_0(\Omega))$ and
\begin{align}
&\Pi_{\mD_l}\beta(u_l) \rightarrow \beta(\bar{u})\quad\text{weakly in } L^2(\Omega_T), \label{eq: weakcon1}\\
&\nabla_{\mD_l}\beta(u_l)\rightarrow \nabla\beta(\bar{u})\quad\text{weakly in } L^2(\Omega_T)^d.\label{eq: weakcon2}
\end{align}
Lemma~\ref{lem: strongconv} also gives, up to a subsequence, the a.e.\ convergence of $\Pi_{\mD_l}\beta(u_l)$ to $\beta(\bar{u})$. As $\beta$ is strictly increasing, this yields
\begin{equation}\label{eq:ae.conv}
\Pi_{\mD_l}u_l\to \bar{u}\quad\text{a.e.\ on $\Omega_T$}.
\end{equation}

Since $\zeta$ is a convex continuous function,
we also deduce from~Lemma~\ref{lem: uniform weak conv} and Lemma~\ref{lem: ap2} that for any $t\in[0,T]$
\begin{equation*}
\int_\Omega \zeta(\bar{u})(t,\x)d\x \leq \liminf_{l\rightarrow\infty} \int_\Omega \zeta(\Pi_{\mD_l}u_l)(t,\x)d\x.
\end{equation*}
Together with~\eqref{eq: priori2} this implies $\zeta(\bar{u})\in L^{\infty}(0,T;L^1(\Omega))$, which shows that $\bar u$ satisfies (i) and (ii) in Definition \ref{def: weak sol}.

\noindent
\underline{Step 2: Passing to the limit in scheme~\eqref{eq: gdmscheme}}.

Let $ \varphi\in C_c^1(-\infty,T)$ and $\psi\in H^{1}_0(\Omega)\cap L^{1+1/m}(\Omega)$. Recalling the definition \eqref{def P} of $P_{\mD_l}$, we take $\phi:= (\varphi(t^{(n-1)})P_{\mD_l}\psi)_{n=0,\ldots,N}$ as test function in~\eqref{eq: gdmscheme} (with $t^{(-1)}=t^{(0)}=0$). 
This gives $T_1^{(l)} + T_2^{(l)} = T_3^{(l)}$ where, dropping the indices $l$ for legibility,
\begin{align*}
T_1^{(l)}
&:=
\sum_{n=0}^{N-1} \varphi(t^{(n)}) \delta t^{(n+\frac12)}\int_{\Omega}\delta_{\mD}^{(n+\frac12)} u(\x)\Pi_\mD P_{\mD}\psi(\x)\,d\x\\
T_2^{(l)}&:= \sum_{n=0}^{N-1} \varphi(t^{(n)}) \delta t^{(n+\frac12)}\int_{\Omega} \Lambda(\x,\Pi_\mD u^{(n+1)}(\x))\nabla_\mD \beta(u^{(n+1)})(\x)\cdot  \nabla_\mD P_{\mD}\psi(\x)\,d\x\\
T_3^{(l)}&:= \sum_{n=0}^{N-1} \varphi(t^{(n)}) \int_{t^{(n)}}^{t^{(n+1)}}\int_{\Omega} f(t,\x)\Pi_\mD P_{\mD}\psi(\x)\,d\x dt.
\end{align*}
Using the following equality (discrete integrate-by-parts, see \cite[Eq.~(D.15)]{Droniou.et.al2018})
\begin{align*}
\sum_{n=0}^{N-1} \varphi(t^{(n)}) \delta^{(n+\frac12)}\delta_{\mD}^{(n+\frac12)} u
&=
\sum_{n=0}^{N-1} \varphi(t^{(n)}) \bigl(\Pi_{\mD} u^{(n+1)} -\Pi_{\mD} u^{(n)}\bigr)\\
&=
\sum_{n=0}^{N-1} \bigl(\varphi(t^{(n+1)})-\varphi(t^{(n)})\bigr)\Pi_{\mD} u^{(n+1)}
-\varphi(0)\Pi_{\mD} u^{(0)},
\end{align*}
we transform $T_1^{(l)}$ into
\begin{align*}
T_1^{(l)} 
= 
-\int_0^T \varphi'(t)\int_{\Omega} \Pi_{\mD} u(t,\x)\Pi_\mD P_{\mD}\psi(\x)\,d\x dt
- \varphi(0)\int_{\Omega} \Pi_{\mD} u^{(0)}(\x)\Pi_\mD P_{\mD}\psi(\x)\,d\x.
\end{align*}
By setting $\varphi_\mD(t):= \varphi(t^{(n)})$ for $t\in (t^{(n)},t^{(n+1)})$, we have
\begin{align*}
T_2^{(l)}
&= \int_0^T \varphi_\mD(t)\int_{\Omega} \Lambda(\x,\Pi_\mD u(t,\x))\nabla_\mD \beta(u)(t,\x)\cdot  \nabla_\mD P_{\mD}\psi(\x)\,d\x dt,\\
T_3^{(l)}
&=\int_0^T \varphi_\mD(t)\int_{\Omega} f(t,\x)\Pi_\mD P_{\mD}\psi(\x)\,d\x dt.
\end{align*}
Since $\varphi_\mD \rightarrow \varphi$ uniformly on $[0,T]$, $\Pi_\mD P_{\mD}\psi\rightarrow \psi$ in $L^2(\Omega)\cap L^{1+1/m}(\Omega)$, $\nabla_\mD P_{\mD}\psi \rightarrow \nabla\psi$ in $L^2(\Omega)^d$,
$\Lambda(\cdot,\Pi_\mD u)\to\Lambda(\cdot,\bar{u})$ a.e.\ on $\Omega_T$ (by continuity of $\Lambda$ with respect to its second argument and \eqref{eq:ae.conv}) and $\Lambda$ is bounded,
letting $l\rightarrow\infty$ in $T_1^{(l)} + T_2^{(l)} = T_3^{(l)}$ we see that $\bar{u}$ satisfies
\begin{gather*}
\begin{align*}
&-\int_0^T \varphi'(t)\int_{\Omega} \bar{u}(t,\x) \psi(\x)\,d\x dt
- \varphi(0)\int_{\Omega}  u_0(\x) \psi(\x)\,d\x\\
&+
\int_0^T \varphi(t)\int_{\Omega} \Lambda(\x,\bar{u}(t,\x))\nabla \beta(\bar{u})(t,\x)\cdot  \nabla\psi(\x)\,d\x dt
=
\int_0^T \varphi(t)\int_{\Omega} f(t,\x)\psi(\x)\,d\x dt.
\end{align*}
\end{gather*}
The above equality also holds with $\varphi(t)\psi(\x)$ replaced by a tensorial function in $C_c^\infty(\Omega_T)$. Hence, from the density of tensorial functions in $L^2(0,T;H^1_0(\Omega))$ \cite[Corollary 1.3.1]{poly} and noting that $\beta(\bar{u})\in L^2(0,T;H^1(\Omega))$ and $f\in L^2(\Omega_T)$, we deduce that $\partial_t\bar{u}$ belongs to $L^{2}(0,T;H^{-1}(\Omega))$ and that $\bar{u}$ satisfies~\eqref{eq: weak sol}.
\end{proof}

\section{The uniform-in-time convergence result}\label{sec:unif_conv}

\begin{lemma}\label{lem: energy eq}
Let $\bar{u}$ be a weak solution of~\eqref{eq: PME}. Then
\begin{enumerate} [label=(\roman*)]
\item for any $T_0\in[0,T]$
\begin{align*}
\int_\Omega \zeta(\bar{u})(T_0,\x)d\x{}&
+\int_0^{T_0}\big\langle\Lambda(\cdot,\bar{u}(t,\cdot))\nabla\beta(\bar{u})(t,\cdot),\nabla\beta(\bar{u})(t,\cdot)\big\rangle_{L^2(\Omega)}dt\\
&=
\int_\Omega \zeta(u_0)(\x)d\x 
+ \int_0^{T_0} \langle f(t),\beta(\bar{u})(t)\rangle_{L^2(\Omega)} dt,
\end{align*}
\item the function 
\[
t\in [0,T]\mapsto \int_\Omega \zeta(\bar{u})(t,\x)d\x \in [0,\infty)
\]
is continuous and bounded,
\item $\bar{u}$ is continuous $[0,T]\rightarrow L^{m+1}(\Omega)$ for the strong topology of $L^{m+1}(\Omega)$.
\end{enumerate}
\end{lemma}
\begin{remark}[About the initial condition]
This lemma is crucial to obtain the uniform-in-time convergence in $L^{m+1}(\Omega)$ of the approximations in Theorem \ref{theo:cv.GS}. If $u_0$ does not belong to $L^{m+1}(\Omega)$, then (i) does not hold any longer (as $\zeta(u_0)$ is not integrable) and the uniform-in-time convergence result cannot be established.
\end{remark}
\begin{proof}

\noindent
\underline{Proof of (i):}
Take $\phi = \mathbf{1}_{(0,T_0)} \beta(\bar{u})\in L^2(0,T;H^1_0(\Omega))$ in~\eqref{eq: weak sol}, where $\mathbf{1}_A$ denotes the characteristic function of $A$. This gives
\begin{multline}\label{eq: uni time1}
\int_0^{T_0}\, _{H^{-1}}\langle \partial_t\bar{u}(t),\beta(\bar{u})(t) \rangle_{H^1_0} dt\\
+
\int_0^{T_0}\int_\Omega\Lambda(\x,\bar{u}(t,\x))\nabla \beta(\bar{u})(t,\x)\cdot
\nabla \beta(\bar{u})(t,\x)\,dtd\x
=\langle f,\beta(\bar{u})\rangle_{L^2(\Omega_{T_0})}.
\end{multline}
It remains to prove that 
\begin{equation}\label{eq: uni time2}
\int_0^{T_0}\, _{H^{-1}}\langle \partial_t\bar{u}(t),\beta(\bar{u})(t) \rangle_{H^1_0} dt 
=
\int_\Omega \zeta(\bar{u})(T_0,\x)d\x - \int_\Omega \zeta(u_0)(\x)d\x.
\end{equation}
Let $\tilde{u}$ be an extension of $\bar{u}$ outside $[0,T_0]$ obtained setting 
\[
\tilde{u}(t)=
\begin{cases}
\bar{u}(T_0)&\quad\text{for } t > T_0\\
\bar{u}(t)&\quad\text{for } t\in[0,T_0].
\end{cases}
\]
Considering the pointwise values of $\bar u$ makes sense owing to the weak continuity of $\bar u:[0,T]\to L^{m+1}(\Omega)$.
We recast the left hand side of~\eqref{eq: uni time2} using the discrete time derivative of $\tilde{u}$, defined by: for $h>0$ and $t>0$,
\[
d_h\tilde{u} (t):=\frac{\tilde{u} (t+h)-\tilde{u}(t)}{h}.
\]
By weak continuity of $\tilde u$ it is easily checked that $\partial_t \tilde u=\mathbf{1}_{(0,T_0)}\partial_t \bar u\in L^2(0,\infty;H^{-1}(\Omega))$. Hence $d_h \tilde u\to \partial_t \tilde u$ weakly in this space as $h\to 0$ and thus
\begin{align}\label{eq: uni time3}
\int_0^{T_0}\, _{H^{-1}}\langle \partial_t\bar{u}(t),{}&\beta(\bar{u})(t) \rangle_{H^1_0} dt
=
\int_0^\infty\, _{H^{-1}}\langle \partial_t\tilde{u}(t),\mathbf{1}_{(0,T_0)}\beta(\bar{u})(t) \rangle_{H^1_0} dt\nonumber\\
&=\lim_{h\rightarrow 0}
\int_0^\infty\, _{H^{-1}}\langle d_h\tilde{u}(t),\mathbf{1}_{(0,T_0)}\beta(\bar{u})(t) \rangle_{H^1_0} dt\nonumber\\
&=\lim_{h\rightarrow 0}\frac{1}{h}
\int_0^{T_0}\int_\Omega \big(\tilde{u}(t+h,\x)-\tilde{u}(t,\x)\big)\beta(\tilde{u})(t,\x)d\x dt.
\end{align}
Since $\beta$ is increasing, $\zeta$ is convex and above its tangent line, which means that $\zeta(b)-\zeta(a)\geq (b-a)\beta(a)$ for all $a,b\in\R$.
Apply this inequality for the right hand side of~\eqref{eq: uni time3} to get
\begin{align}\label{eq: uni time4}
\int_0^{T_0}\, _{H^{-1}}\langle &\partial_t\bar{u}(t),\beta(\bar{u})(t) \rangle_{H^1_0} dt\nonumber\\
&\leq\lim_{h\rightarrow 0}\frac{1}{h}
\int_0^{T_0}\int_\Omega \zeta(\tilde{u})(t+h,\x)-\zeta(\tilde{u})(t,\x)d\x dt\nonumber\\
&=\lim_{h\rightarrow 0}\frac{1}{h}\bigg[
\int_{T_0}^{T_0+h}\int_\Omega \zeta(\tilde{u})(t,\x)d\x dt - \int_{0}^{h}\int_\Omega \zeta(\tilde{u})(t,\x)d\x dt \bigg]\nonumber\\
&\leq \int_\Omega \zeta(\bar{u})(T_0,\x)d\x - \liminf_{h\rightarrow 0}\frac{1}{h}\int_{0}^{h}\int_\Omega \zeta(\bar{u})(t,\x)d\x dt.
\end{align}
Since $\bar{u}\in C([0,T];L^{m+1}(\Omega)_{\mathrm{w}})$, we have
\[
\frac{1}{h}\int_{0}^{h} \bar{u}(t)dt\rightarrow \bar{u}(0)\,\text{weakly in } L^{m+1}(\Omega)\,\text{as } h\rightarrow 0.
\]
This together with the convexity of $\zeta$ and Jensen's inequality gives
\begin{align}\label{eq: uni time5}
\int_\Omega \zeta(\bar{u})(0,\x)d\x 
&\leq 
\liminf_{h \rightarrow 0}\int_\Omega \zeta\left(\frac{1}{h}\int_{0}^{h}\bar{u}(t,\x)dt\right)d\x\nonumber\\
&\leq 
\liminf_{h \rightarrow 0} \int_\Omega \frac{1}{h}\int_{0}^{h}\zeta(\bar{u})(t,\x)dt d\x.
\end{align}
It follows from~\eqref{eq: uni time4} and~\eqref{eq: uni time5} that 
\begin{equation}\label{eq: uni time6}
\int_0^{T_0}\, _{H^{-1}}\langle \partial_t\bar{u}(t),\beta(\bar{u})(t) \rangle_{H^1_0} dt
\leq 
\int_\Omega \zeta(\bar{u})(T_0,\x)d\x -\int_\Omega \zeta(\bar{u})(0,\x)d\x .
\end{equation}
The reverse inequality of~\eqref{eq: uni time6} is obtained by reversing the time. Indeed, applying~\eqref{eq: uni time6} with 
$\bar{u}$ replaced by $v(t):= \bar{u}(T_0-t)$ for $t\in[0,T_0]$ and noting that $\partial_t v(t)=-\partial_t \bar{u}(T_0-t)$, we deduce
\begin{equation}\label{eq: uni time7}
\int_0^{T_0}\, _{H^{-1}}\langle \partial_t\bar{u}(t),\beta(\bar{u})(t) \rangle_{H^1_0} dt
\geq
\int_\Omega \zeta(\bar{u})(T_0,\x)d\x -\int_\Omega \zeta(\bar{u})(0,\x)d\x .
\end{equation}
Hence, equality~\eqref{eq: uni time2} follows immediately from~\eqref{eq: uni time6} and~\eqref{eq: uni time7}, which completes the proof of (i).

\noindent
\underline{Proof of (ii) and (iii):}
The continuity and boundedness of 
\[
t\in [0,T]\mapsto \int_\Omega \zeta(\bar{u})(t,\x)d\x \in [0,\infty)
\]
is straightforward consequence of (i). 

Let $s\in [0,T]$ and $(s_n)_{n\geq 1}$ be a sequence in $[0,T]$ converging to $s$. Since $\zeta(z)=\frac{1}{m+1}|z|^{m+1}$, it follows from (ii) and the continuity of $\bar{u} : [0,T]\rightarrow L^{m+1}(\Omega)$ for the weak topology of $L^{m+1}(\Omega)$ stated in Lemma~\ref{lem: uniform weak conv} that 
\begin{align*}
&\lim_{n\rightarrow \infty} \|\bar{u}(s_n)\|_{L^{m+1}(\Omega)} = \|\bar{u}(s)\|_{L^{m+1}(\Omega)}\mbox{ and }\\
&\bar{u}(s_n)\rightarrow \bar{u}(s)\text{ weakly in } L^{m+1}(\Omega)\text{ as } n\rightarrow\infty.
\end{align*}
This implies the strong convergence $\bar{u}(s_n)\rightarrow \bar{u}(s)$ in $L^{m+1}(\Omega)$. 
Hence,  $\bar{u}$ is continuous $[0,T]\rightarrow L^{m+1}(\Omega)$, which complete the proof of the lemma.
\end{proof}

We can now prove our main convergence result.
\medskip

\begin{proof}[ of Theorem \ref{theo:cv.GS}]

We consider the subsequence provided by Theorem \ref{theo:cv.GS:weak}.

\medskip

\underline{Proof of the uniform-in-time convergence}.

Let $T_0\in [0,T]$ and $(T_l)_{l\geq 1}$ be a sequence in $[0,T]$ converging to $T_0$. We note that for any $l\geq 1$ there exists an integer number $k\in [1,N]$ such that $T_l\in (t^{(k-1)},t^{(k)}]$. It follows from~\eqref{eq: priori1} that 
\begin{align}\label{eq: uni time conv1}
\int_\Omega \zeta(\Pi_{\mD_l} u_l)(T_l,\x)d\x
+ &\big\langle \Lambda(\cdot,\Pi_{\mD_l} u)\nabla_{\mD_l}\beta(u),
\nabla_{\mD_l}\beta(u)\big\rangle_{\Omega_{T_l}}
\nonumber\\
&\leq\int_\Omega\zeta(\Pi_{\mD_l} u_l^{(0)})(\x)d\x 
+\big\langle f,  \Pi_{\mD_l} \beta(u_l)\big\rangle_{L^2(\Omega_{t^{(k)}})}.
\end{align}
Moving the term involving $\nabla_{\mD_l}\beta(u_l)$ to the right-hand side,
we take the supremum limit as $l$ tends to infinity of the above inequality. Reasoning as in \cite[Proof of Eq.~(48)]{DHM16}, the weak convergence of $\nabla_{\mD_l}\beta(u_l)$ to $\nabla\beta(\bar u)$, the strong convergence in $L^2(0,T)$ of $\mathbf{1}_{(0,T_l)}$ to $\mathbf{1}_{(0,T_0)}$, and the a.e.\ convergence of $\Lambda(\cdot,\Pi_{\mD_l}u_l)$ to $\Lambda(\cdot,\bar{u})$ (see the proof of Theorem \ref{theo:cv.GS:weak}) yield
\begin{align*}
-\big\langle \Lambda(\cdot,\bar{u})\nabla\beta(\bar{u}),\nabla\beta(\bar{u})\big\rangle_{\Omega_{T_0}}
\ge{}& -\liminf_{l\rightarrow\infty}\big\langle\Lambda(\cdot,\Pi_{\mD_l}u_l)\nabla_{\mD_l}\beta(u_l),\nabla_{\mD_l}\beta(u_l)\big\rangle_{\Omega_{T_l}}\\
={}& \limsup_{l\rightarrow\infty} (-\big\langle \Lambda(\cdot,\Pi_{\mD_l}u_l)\nabla_{\mD_l}\beta(u_l),\nabla_{\mD_l}\beta(u_l)\big\rangle_{\Omega_{T_l}}).
\end{align*} 

We thus deduce
\begin{align*}
\limsup_{l\rightarrow\infty} \int_\Omega{}& \zeta(\Pi_{\mD_l} u_l)(T_l,\x)d\x
\\
\leq{}& -\big\langle \Lambda(\cdot,\bar{u})\nabla\beta(\bar{u}),\nabla\beta(\bar{u})\big\rangle_{\Omega_{T_0}}
+\limsup_{l\rightarrow\infty} \int_\Omega\zeta(\Pi_{\mD_l} u_l^{(0)})(\x)d\x\\
{}&+ \limsup_{l\rightarrow\infty} \big\langle f,  \Pi_{\mD_l} \beta(u_l)\big\rangle_{L^2(\Omega_{t^{(k)}})}.
\end{align*}
As $l\to\infty$ we have $t^{(k)}\to T_0$ since $|T_l-t^{(k)}|\le \delta t_{\mD_l}$.
Together with~\eqref{eq: weakcon1} and the convergence $\Pi_{\mD_l} \mI_{D_l} u_0\rightarrow u_0$ in $L^{m+1}(\Omega)$ (consistency of $(\mD_l)_{l\ge 1}$), this implies
\begin{multline*}
\limsup_{l\rightarrow\infty} \int_\Omega \zeta(\Pi_{\mD_l} u_l)(T_l,\x)d\x\\
\leq 
\int_\Omega\zeta( u_0)(\x)d\x
+ \big\langle f,   \beta(\bar{u})\big\rangle_{L^2(\Omega_{T_0})}
-\big\langle \Lambda(\cdot,\bar{u})\nabla\beta(\bar{u}),\nabla\beta(\bar{u})\big\rangle_{\Omega_{T_0}}.
\end{multline*}
Using the energy equality in Lemma~\ref{lem: energy eq} we infer
\begin{equation*}
\limsup_{l\rightarrow\infty} \int_\Omega \zeta(\Pi_{\mD_l} u_l)(T_l,\x)d\x
\leq
\int_\Omega \zeta( \bar{u})(T_0,\x)d\x,
\end{equation*}
or equivalently, since $\zeta(z)=\frac{1}{m+1}|z|^{m+1}$,
\begin{equation*}
\limsup_{l\rightarrow\infty} \|\Pi_{\mD_l} u_l(T_l)\|_{L^{m+1}(\Omega)}
\leq
\| \bar{u}(T_0)\|_{L^{m+1}(\Omega)}.
\end{equation*}
By Lemma~\ref{lem: uniform weak conv} and the uniformly convexity of the space $L^{m+1}(\Omega)$, this implies the strong convergence 
$\Pi_{\mD_l}  u_l(T_l) \rightarrow \bar{u}(T_0)$ in $L^{m+1}(\Omega)$. 
The convergence of $(\Pi_{\mD_l}u_l)_{l\ge 1}$ in $L^\infty(0,T;L^{m+1}(\Omega))$ then follows by \cite[Lemma C.13]{Droniou.et.al2018}.

\medskip

\underline{Proof of the strong convergence of gradients}.

By taking the supremum limit as $l\rightarrow \infty$ of~\eqref{eq: uni time conv1} with $T_l = T$ and noting from the $L^\infty(0,T;L^{m+1}(\Omega))$ convergence of $(\Pi_{\mD_l}u_l)_{l\ge 1}$ that 
\[
\lim_{l\rightarrow\infty} \int_\Omega \zeta(\Pi_{\mD_l} u_l)(T,\x)d\x = \int_\Omega \zeta(\bar{u})(T,\x)d\x,
\]
we obtain
\begin{multline*}
\limsup_{l\rightarrow\infty} \big\langle\Lambda(\cdot,\Pi_{\mD_l}u_l)\nabla_{\mD_l}\beta(u_l),
\nabla_{\mD_l}\beta(u_l)\big\rangle_{\Omega_T}
\\
\leq 
\int_\Omega\zeta( u_0)(\x)d\x
+ \big\langle f,   \beta(\bar{u})\big\rangle_{L^2(\Omega_{T})}
-\int_\Omega \zeta(\bar{u})(T,\x)d\x.
\end{multline*}
Together with the energy equality in Lemma~\ref{lem: energy eq} for $T_0=T$, this implies
\[
\limsup_{l\rightarrow\infty} \big\langle\Lambda(\cdot,\Pi_{\mD_l}u_l)\nabla_{\mD_l}\beta(u_l),
\nabla_{\mD_l}\beta(u_l)\big\rangle_{\Omega_T}
\leq \big\langle\Lambda(\cdot,\bar{u})\nabla\beta(\bar{u}),\nabla\beta(\bar{u})\big\rangle_{\Omega_T}.
\]
The strong convergence of $\nabla_{\mD_l}\beta(u_l)$ follows from the above inequality, the weak convergence~\eqref{eq: weakcon2}, and the a.e.\ convergence $\Lambda(\cdot,\Pi_{\mD_l}u_l)\to \Lambda(\cdot,\bar{u})$, by writing
\begin{multline*}
\underline{\lambda}\|\nabla_{\mD_l}\beta(u_l)-\nabla\beta(\bar{u})\|^2_{L^2(\Omega_T)}
\\
\le \big\langle \Lambda(\cdot,\Pi_{\mD_l}u_l)(\nabla_{\mD_l}\beta(u_l)-\nabla\beta(\bar{u})),
\nabla_{\mD_l}\beta(u_l)-\nabla\beta(\bar{u})\big\rangle_{\Omega_T},
\end{multline*}
by developing the right-hand side, and by taking the superior limit.
\end{proof}
\section{Numerical results}\label{sec:tests}

\subsection{Setting}

In this section, we present numerical results based on two different schemes: the mass-lumped $\mathbb{P}^1$ finite elements (ML$\mathbb{P}^1$ for short) and the Hybrid Mimetic Mixed method (HMM for short). This latter method is a polytopal method, which means that it can be applied to meshes made of general polygons/polyhedras; it is actually a generalisation of the Hybrid Finite Volume (SUSHI) method, the Mimetic Finite Difference method and the Mixed Finite Volume scheme \cite{DEGH09,DE06,sushi,mfdrev}. As shown in \cite[Section 8.4 and Chapter 13]{Droniou.et.al2018}, each of ML$\mathbb{P}^1$ and HMM is a GDM for a certain choice of gradient discretisations, and our previous analysis therefore applies to them. MATLAB codes for these schemes are available at \texttt{https://github.com/jdroniou/matlab-PME}. 

The accuracy of the schemes will be measured against the analytical solution to~\eqref{eq: PME}, with $\Lambda={\rm Id}$ and $f=0$, called the Barenblatt solution~\cite{Pattle1959,Barenblatt52}, which has the following explicit form
\begin{equation}\label{eq: Barenblatt}
u_B(t,x) = t^{-\alpha}[C_B-\gamma |x|^2t^{-2\rho}]_{+}^{1/(m-1)},
\end{equation}
where $[s]_+ = \max\{s,0\}$, $\alpha = \frac{d}{d(m-1)+2}$, $\rho=\frac{\alpha}{d}$, $\gamma = \frac{\alpha(m-1)}{2md}$ and $C_B$ is an arbitrary strictly positive constant.

We consider the domain $\Omega=(-0.5;0.5)\times(-0.5;0.5)$ and the final time $T=1$. The Barenblatt solution is singular at $t=0$ so we run the simulations with the initial condition $u_B(t_0,\cdot)$ for some $t_0>0$.  For $m>1$, we take $t_0=0.1$, and $C_B=0.005$ is small enough so that $u_B$ vanishes on $\partial\Omega$ over the considered time interval. For $m<1$, the Barenblatt solution is much steeper for times around 0.1, and the meshes we consider are not small enough to capture the extreme spike of the solution at $x=0$; we therefore take $t_0=0.5$, and $C_B=0.1$. The boundary conditions are adjusted in the scheme to match the exact values of the $u_B$ on $\partial\Omega$ at the considered times.

Solving the non-linear system resulting from a scheme for the porous medium equation can be difficult, as Newton iterations do not necessarily behave very well \cite{RPK06,LR16}. Some remedy, consisting in re-parametrising the system, can be put in place \cite{BC17}. At the mesh sizes and time steps we consider in our tests, however, we found that a simple Newton algorithm works very well in both regimes. In case of fast diffusion, to avoid the singularity coming from $|u|^{m-1}u$ when $m<1$, the algebraic system is written on the unknown $v=|u|^{m-1}u$; this creates a non-linear reaction term $|v|^{\frac1m-1}v$, which poses no issues to the Newton algorithm since $1/m>1$. In all our tests, we found that 1 to 3 Newton iterations only were necessary to reach a residual in the non-linear system of $10^{-8}$ or less (in supremum norm).

\subsection{Rates of convergence with respect to the mesh size}\label{sec:rates.h}

The tests presented in this section have been run using families of triangular meshes (for ML$\mathbb{P}^1$) or hexagonal meshes (for HMM); the first two elements of each family are presented in Figure \ref{fig:hexa}. We chose a uniform time step $\delta t^{(n+\frac12)}= h^2$ to ensure that the leading truncation error was due to the spatial discretisation.

\begin{figure}
\begin{tabular}{cccc}
\includegraphics[width=0.3\textwidth]{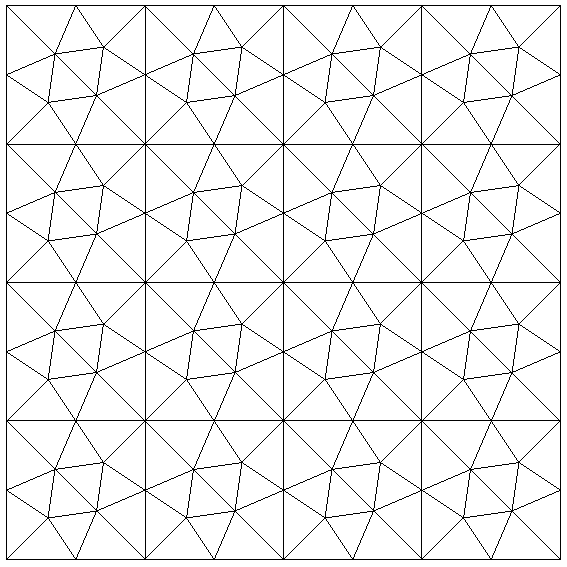}& 
\includegraphics[width=0.3\textwidth]{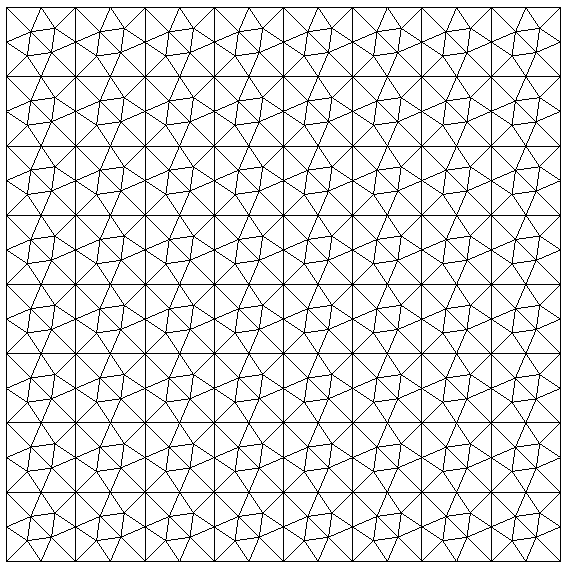}\\
\includegraphics[width=0.3\textwidth]{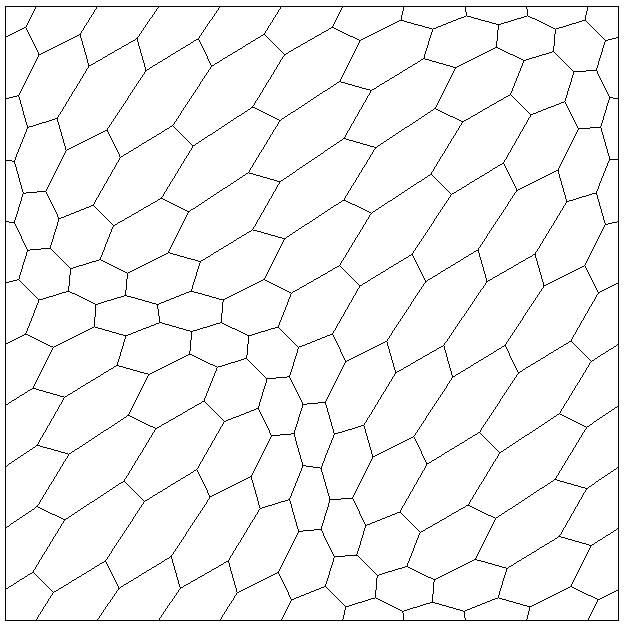}& 
\includegraphics[width=0.3\textwidth]{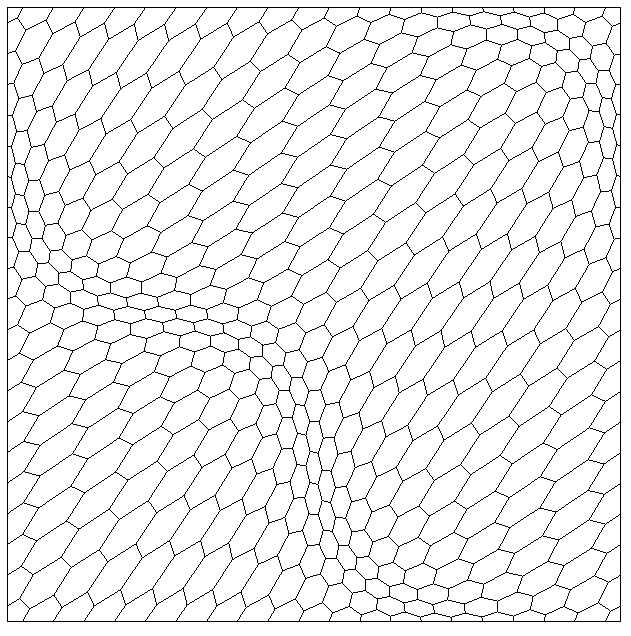}
\end{tabular}
\caption{First two elements in each mesh family used for ML$\mathbb{P}^1$ (top) and HMM (bottom).}
\label{fig:hexa}
\end{figure}

The accuracy of the schemes is assessed by considering the following relative errors on $u$ and $\beta(u)$:
\begin{equation}\label{eq:errors}
\begin{aligned}
E_{u,L^{m+1}}:={}& \frac{\|u_B(t_0+T) - u^{(N)}\|_{L^{m+1}}}{\|u_B(t_0+T)\|_{L^{m+1}}}\quad\mbox{ and }\\
E_{\beta(u),L^2}:={}& \frac{\|\beta(u_B)(t_0+T) - \beta(u^{(N)})\|_{L^{2}}}{\|\beta(u_B)(t_0+T)\|_{L^{2}}}.
\end{aligned}
\end{equation}
Tables \ref{tab:P1_slow_Lmp1}--\ref{tab:HMM_slow_beta} show the results obtained for various values of $m>1$ in the slow diffusion case. The rate of convergence in $L^2$-norm for the error in the (regular) variable $\beta(u)$ is uniformly close to two, for both schemes and irrespective of the value of $m$ (the magnitudes of the errors also do not seem to depend strongly on $m$). On the contrary, for the less regular variable $u$, the rate of convergence in $L^{m+1}$-norm seem to decrease as $m$ increases. To our knowledge, no results in the literature provides rates of convergence for the errors assessed at the final time, as in \eqref{eq:errors}. The only available rates of convergence seem to be for time-space averaged quantities, the closest result being the ones in \cite{NochettoVerdi88} for the conforming $\mathbb{P}^1$ approximation, and with a regularisation parameter (due to approximating the Kirchhoff-tranformed equation, instead of directly considering the porous medium equation), which is not needed in our scheme. With an optimal choice of the regularisation parameter and time steps the error estimates are $\mathcal O(h^{2m/(3m-1)})$ in $L^2(\Omega\times(0,T))$-norm on $\beta(u)$, and $\mathcal O(h^{4m/[(m+1)(3m-1)]})$ in $L^{m+1}(\Omega\times(0,T))$-norm on $u$. For $m\in\{1.5,2,2.5,3\}$, these rates are respectively in the range 0.75--0.85 and 0.38--0.68 (going up as $m$ decreases towards 1). These are rather pessimistic predictions compared to the results presented here. The main reason that could explain this is that they come from an analysis that mixes several norms, including an energy $H^1$-norm on time-integrated errors, which might not lead to optimal rates for errors in Lebesgue norms. The only clear common behaviour between our numerically estimated rates and these theoretical rates is that, when considering errors on $u$, all decrease as $m$ increases. Note however that the $L^{m+1}$-norm becomes more stringent as $m$ increases, which contributes to explaining the reduced rates in this case. We also mention the recent work \cite{DE19}, which provides (possibly) high-order estimates for the stationary Stefan/porous medium equations; the tests provided in this reference also show, in the time-independent setting, that higher-than-expected rates can be obtained in Lebesgue norms. An extension of these results to time-dependent models is the subject of future work.

\begin{table}[h!]
\centering
\begin{tabular}{c| c  c | c  c | c c | c c  }
  & $m=1.5$ & & $m=2.0$ & & $m=2.5$ &  & $m=3.0$ &\\
 $h $ &  $E_{u,L^{m+1}}$ & rate  &  $E_{u,L^{m+1}}$ & rate &  $E_{u,L^{m+1}}$ & rate  &  $E_{u,L^{m+1}}$ & rate  \\
\hline
$0.25$& $1.80\eminus{}00$ & & $6.82\eminus{}01$ &  & $3.80\eminus{}01$ & & $5.21\eminus{}01$ & \\
$0.13$& $1.71\eminus{}01$ & $3.40$ & $8.38\eminus{}02$& $3.03$ & $1.50\eminus{}01$& $1.35$ & $8.49\eminus{}02$& $2.62$\\
$0.06$& $2.49\eminus{}02$  & $2.78$&$2.97\eminus{}02$ & $1.50$& $5.68\eminus{}02$& $1.40$ & $7.09\eminus{}02$& $0.26$\\
$0.03$& $6.10\eminus{}03$ & $2.03$&$1.30\eminus{}02$ & $1.19$& $2.84\eminus{}02$& $1.00$ & $3.95\eminus{}02$& $0.85$
\end{tabular}
\caption{ML$\mathbb{P}^1$: Errors $E_{u,L^{m+1}}$ and convergence rates w.r.t.\ to the mesh size $h$, for uniform time steps $\delta t^{(n+\frac12)}= h^2$.}
\label{tab:P1_slow_Lmp1}
\end{table}

\begin{table}[h!]
\centering
\begin{tabular}{c| c  c | c  c | c c | c c  }
  & $m=1.5$ & & $m=2.0$ & & $m=2.5$ &  & $m=3.0$ &\\
 $h $ &  $E_{\beta(u),L^2}$ & rate  &  $E_{\beta(u),L^2}$ & rate &  $E_{\beta(u),L^2}$ & rate  &  $E_{\beta(u),L^2}$ & rate  \\
\hline
$0.25$& $3.72\eminus{}00$ & & $1.90\eminus{}00$ &  & $9.97\eminus{}01$ & & $6.35\eminus{}01$ & \\
$0.13$& $2.43\eminus{}01$ & $3.93$ & $1.46\eminus{}01$& $3.70$ & $1.09\eminus{}01$& $3.19$ & $1.49\eminus{}01$& $2.09$\\
$0.06$& $2.68\eminus{}02$  & $3.18$&$3.83\eminus{}02$ & $1.93$& $2.65\eminus{}02$& $2.04$ & $3.29\eminus{}02$& $2.18$\\
$0.03$& $6.84\eminus{}03$ & $1.97$&$6.58\eminus{}03$ & $2.54$& $7.79\eminus{}03$& $1.77$ & $8.49\eminus{}03$& $1.95$
\end{tabular}
\caption{ML$\mathbb{P}^1$: Errors $E_{\beta(u),L^2}$ and convergence rates w.r.t.\ to the mesh size $h$, for uniform time steps $\delta t^{(n+\frac12)}= h^2$.}
\label{tab:P1_slow_beta}
\end{table}


\begin{table}[h!]
\centering
\begin{tabular}{c| c  c | c  c | c c | c c  }
  & $m=1.5$ & & $m=2.0$ & & $m=2.5$ &  & $m=3.0$ &\\
 $h $ &  $E_{u,L^{m+1}}$ & rate  &  $E_{u,L^{m+1}}$ & rate &  $E_{u,L^{m+1}}$ & rate  &  $E_{u,L^{m+1}}$ & rate  \\
\hline
$0.24$& $1.00\eminus{}01$ & & $8.77\eminus{}02$ &  & $9.23\eminus{}02$ & & $1.58\eminus{}01$ & \\
$0.13$& $3.27\eminus{}02$ & $1.81$ & $3.65\eminus{}02$& $1.41$ & $4.96\eminus{}02$& $1.00$ & $8.25\eminus{}02$& $1.05$\\
$0.07$& $1.02\eminus{}02$  & $1.72$&$1.59\eminus{}02$ & $1.22$& $2.59\eminus{}02$& $0.95$ & $3.14\eminus{}02$& $1.42$\\
$0.03$& $2.93\eminus{}03$ & $1.80$&$5.18\eminus{}03$ & $1.63$& $1.17\eminus{}02$& $1.16$ & $2.26\eminus{}02$& $0.47$
\end{tabular}
\caption{HMM: Errors $E_{u,L^{m+1}}$ and convergence rates w.r.t.\ to the mesh size $h$, for uniform time steps $\delta t^{(n+\frac12)}= h^2$.}
\label{tab:HMM_slow_Lmp1}
\end{table}

\begin{table}[h!]
\centering
\begin{tabular}{c| c  c | c  c | c c | c c  }
  & $m=1.5$ & & $m=2.0$ & & $m=2.5$ &  & $m=3.0$ &\\
 $h $ &  $E_{\beta(u),L^2}$ & rate  &  $E_{\beta(u),L^2}$ & rate &  $E_{\beta(u),L^2}$ & rate  &  $E_{\beta(u),L^2}$ & rate  \\
\hline
$0.24$& $1.23\eminus{}01$ & & $1.23\eminus{}01$ &  & $1.00\eminus{}01$ & & $8.35\eminus{}02$ & \\
$0.13$& $3.97\eminus{}02$ & $1.83$ & $3.78\eminus{}02$& $1.90$ & $2.90\eminus{}02$& $1.99$ & $5.00\eminus{}02$& $0.83$\\
$0.07$& $1.09\eminus{}02$  & $1.90$&$9.81\eminus{}03$ & $1.99$& $9.94\eminus{}03$& $1.58$ & $1.13\eminus{}02$& $2.19$\\
$0.03$& $2.78\eminus{}03$ & $1.98$&$2.53\eminus{}03$ & $1.96$& $2.68\eminus{}03$& $1.90$ & $2.92\eminus{}03$& $1.96$
\end{tabular}
\caption{HMM: Errors $E_{\beta(u),L^2}$ and convergence rates w.r.t.\ to the mesh size $h$, for uniform time steps $\delta t^{(n+\frac12)}= h^2$.}
\label{tab:HMM_slow_beta}
\end{table}


The results in the fast diffusion case $m<1$ are presented in Tables \ref{tab:P1_fast_Lmp1}--\ref{tab:HMM_fast_beta}.
For ML$\mathbb{P}^1$ we observe a rate of convergence of at least 1 in both $u$ and $\beta(u)$. HMM behaves better, with a rate of at least 1.6, and getting close to 2 as $m$ increases towards one. A possible explanation for this difference of behaviour is that, for the considered meshes and a given size $h$, HMM has more degrees of freedom that ML$\mathbb{P}^1$.

\begin{table}[h!]
\centering
\begin{tabular}{c| c  c | c  c | c c   }
  & $m=0.3$ & & $m=0.5$ & & $m=0.7$ &\\
 $h $ &  $E_{u,L^{m+1}}$ & rate  &  $E_{u,L^{m+1}}$ & rate &  $E_{u,L^{m+1}}$ & rate   \\
\hline
$0.25$& $3.85\eminus{}00$ &       & $2.93\eminus{}01$ &      & $1.40\eminus{}01$ &  \\
$0.13$& $3.52\eminus{}01$ & $3.45$ & $7.65\eminus{}02$& $1.94$ & $5.62\eminus{}02$& $1.32$  \\
$0.06$& $5.88\eminus{}02$  &$2.58$ &$3.36\eminus{}02$ & $1.19$& $2.50\eminus{}02$& $1.17$ \\
$0.03$& $2.58\eminus{}02$  &$1.19$ &$1.56\eminus{}02$ & $1.10$& $1.17\eminus{}02$& $1.09$
\end{tabular}
\caption{ML$\mathbb{P}^1$: Errors $E_{u,L^{m+1}}$ and convergence rates w.r.t.\ to the mesh size $h$, for uniform time steps $\delta t^{(n+\frac12)}= h^2$.}
\label{tab:P1_fast_Lmp1}
\end{table}

\begin{table}[h!]
\centering
\begin{tabular}{c| c  c | c  c | c c   }
  & $m=0.3$ & & $m=0.5$ & & $m=0.7$ &\\
 $h $ &  $E_{\beta(u),L^2}$ & rate  &  $E_{\beta(u),L^2}$ & rate &  $E_{\beta(u),L^2}$ & rate  \\
\hline
$0.24$& $4.52\eminus{}01$ &  & $1.34\eminus{}01$ & & $9.56\eminus{}02$ & \\
$0.13$& $8.81\eminus{}02$ & $2.36$ & $3.71\eminus{}02$& $1.85$ & $3.89\eminus{}02$& $1.30$  \\
$0.07$& $1.68\eminus{}03$  &$2.39$ &$1.66\eminus{}02$ & $1.16$& $1.74\eminus{}02$& $1.16$ \\
$0.03$& $7.60\eminus{}03$  &$1.15$ &$7.82\eminus{}03$ & $1.09$& $8.21\eminus{}03$& $1.09$
\end{tabular}
\caption{ML$\mathbb{P}^1$: Errors $E_{\beta(u),L^2}$ and convergence rates w.r.t.\ to the mesh size $h$, for uniform time steps $\delta t^{(n+\frac12)}= h^2$.}
\label{tab:P1_fast_beta}
\end{table}

\begin{table}[h!]
\centering
\begin{tabular}{c| c  c | c  c | c c   }
  & $m=0.3$ & & $m=0.5$ & & $m=0.7$ &\\
 $h $ &  $E_{u,L^{m+1}}$ &  rate  &  $E_{u,L^{m+1}}$ & rate &  $E_{u,L^{m+1}}$ & rate \\
\hline
$0.24$& $1.18\eminus{}01$ &  & $3.03\eminus{}02$ & & $2.58\eminus{}02$ & \\
$0.13$& $1.48\eminus{}02$ & $3.34$ & $9.31\eminus{}03$& $1.90$ & $7.82\eminus{}03$& $1.92$  \\
$0.07$& $3.73\eminus{}03$  &$2.03$ &$2.46\eminus{}03$ & $1.96$& $2.05\eminus{}03$& $1.97$ \\
$0.03$& $1.23\eminus{}03$  &$1.60$ &$6.97\eminus{}04$ & $1.83$& $5.21\eminus{}04$& $1.99$
\end{tabular}
\caption{HMM: Errors $E_{u,L^{m+1}}$ and convergence rates w.r.t.\ to the mesh size $h$, for uniform time steps $\delta t^{(n+\frac12)}= h^2$.}
\label{tab:HMM_fast_Lmp1}
\end{table}

\begin{table}[h!]
\centering
\begin{tabular}{c| c  c | c  c | c c   }
  & $m=0.3$ & & $m=0.5$ & & $m=0.7$ &\\
 $h $ &  $E_{\beta(u),L^2}$ &  rate  &  $E_{\beta(u),L^2}$ & rate &  $E_{\beta(u),L^2}$ & rate \\
\hline
$0.24$& $3.29\eminus{}02$ &  & $1.52\eminus{}02$ & & $1.81\eminus{}02$ & \\
$0.13$& $4.41\eminus{}03$ & $3.24$ & $4.67\eminus{}03$& $1.90$ & $5.50\eminus{}03$& $1.92$  \\
$0.07$& $1.11\eminus{}03$  &$2.02$ &$1.24\eminus{}03$ & $1.95$& $1.45\eminus{}03$& $1.97$ \\
$0.03$& $3.66\eminus{}04$  &$1.61$ &$3.50\eminus{}04$ & $1.83$& $3.67\eminus{}04$& $1.99$
\end{tabular}
\caption{HMM: Errors $E_{\beta(u),L^2}$ and convergence rate w.r.t.\ to the mesh size $h$, for uniform time steps $\delta t^{(n+\frac12)}= h^2$.}
\label{tab:HMM_fast_beta}
\end{table}

\subsection{Rates with respect to the time step, and errors for front distances}

The tests presented here were obtained by implementing ML$\mathbb{P}^1$ in Python 2.7.12 using FEniCS 1.5~\cite{Fenics}, and the embedded mesh generator. We first present, in Table~\ref{tab: 1b}, the error $E_{u,L^{m+1}}$ for several values of uniform time steps $\delta t^{(n+\frac12)}$ and a fixed space step $h= 2^{-7}$. For moderate $m$, the rates in this table are close to $1$, which is expected given that we use an implicit first-order time stepping. The decay of convergence rate for higher $m$ is due to a saturation of the spatial errors.

\begin{table}[h!]
\centering
\begin{tabular}{c | c  c | c  c | c c | c c  }
 & $m=1.5$ & & $m=2.0$ & & $m=2.5$ & & $m=3$ & \\
 $k $& $E_{u,L^{m+1}}$ & rate &  $E_{u,L^{m+1}}$ & rate &  $E_{u,L^{m+1}}$ & rate & $E_{u,L^{m+1}}$ & rate \\
\hline
$1/4$& $9.73\eminus{}02$ & & $8.15\eminus{}02$ & & $7.18\eminus{}02$ & &$6.72\eminus{}02$& \\
$1/8$& $4.9\eminus{}02$ & $0.99$ & $4.19\eminus{}02$& $0.96$ & $3.77\eminus{}02$& $0.93$ & $3.68\eminus{}02$ &  $0.87$\\
$1/16$& $2.45\eminus{}02$ & $1.00$ &$2.12\eminus{}02$ & $0.98$ & $1.95\eminus{}02$& $0.95$ & $1.99\eminus{}02$& $0.89$\\
$1/32$& $1.21\eminus{}02$ & $1.02$& $1.06\eminus{}02$ & $1.0$  & $1.01\eminus{}02$& $0.95$ & $1.32\eminus{}02$& $0.59$\\
$1/64$& $5.82\eminus{}03$ &$1.06$ &  $5.28\eminus{}03$ & $1.01$& $6.56\eminus{}03$ & $0.62$ & $1.29\eminus{}02$& $0.03$
\end{tabular}
\caption{ML$\mathbb{P}^1$: Errors $E_{u,L^{m+1}}$ and convergence rates w.r.t.\ the time discretisation.}
\label{tab: 1b}
\end{table}

Finally, in Table~\ref{ta: 2}, we compute the front distances at the final time $T$ of the exact and approximate solutions, with fixed space and time steps $h= \delta t^{(n+\frac12)}= 2^{-7}$. In this table, we set
\[
d_v:= \max_{x\in \Omega}\left\{|x|\,: \, v(x) \neq 0\right\}.
\] 
The relative errors for these front distances are also provided. The front is relatively well approximated, despite the usage of a low-order mass-lumped method. The error increases with $m$, which is consistent with the results in Section \ref{sec:rates.h}, and indicates that the Barenblatt solution is more challenging to approximate for larger values of $m$ -- probably due to its reduced regularity as $m$ increases.

The surface plots of the numerical solution for $m=2.5$, using $\delta t^{(n+\frac12)}=10^{-3}$ and $h =2^{-7}$, are presented in Figure~\ref{fig:2}. We notice the preservation of symmetry of the solution, and the expected expansion combined with diminution of the maximal value of the solution. 

\begin{table}[h!]
\centering
\begin{tabular}{c| c  c | c  }
 $m $ &  $d_u$  & $d_{u_B}$ &  $\frac{|d_u - d_{u_B}|}{d_{u_B}}$  \\
\hline
$2.0$& $0.401$ & $0.404$ & $0.8\%$ \\
$2.2$& $0.408$ & $0.406$ & $0.6\%$ \\
$2.5$& $0.422$ & $0.412$ & $2.25\%$\\
$2.7$& $0.431$ & $0.418$ & $3.0\%$ \\
$3.0$& $0.446$ & $0.428$ & $4.11\%$
\end{tabular}
\caption{ML$\mathbb{P}^1$: Front distance of the approximate solution and the Barenblatt solution.}
\label{ta: 2}
\end{table}

\begin{figure}
\subfloat[$t= 0.1$]{\includegraphics[width=0.5\textwidth, height=0.3\textheight]{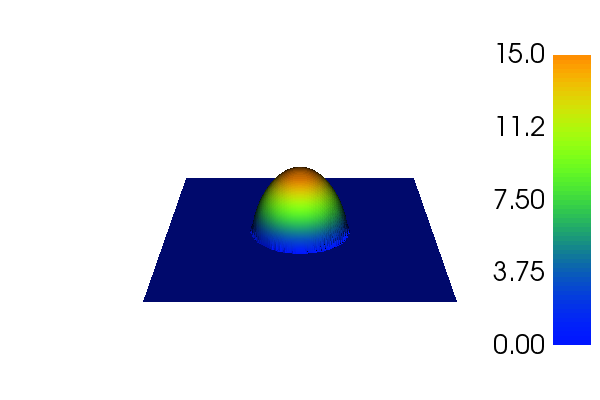}} 
\subfloat[$t= 0.19$]{\includegraphics[width=0.5\textwidth, height=0.3\textheight]{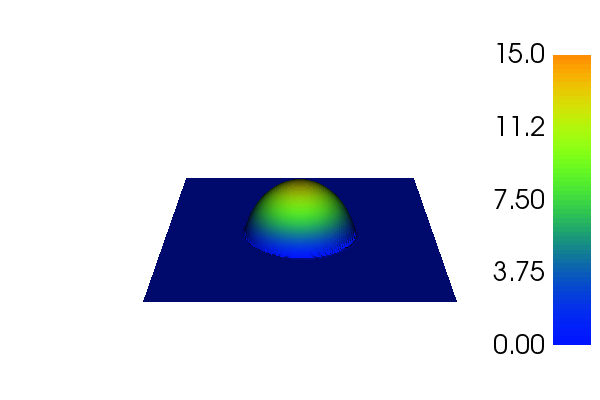}}\\
\subfloat[$t= 0.37$]{\includegraphics[width=0.5\textwidth, height=0.3\textheight]{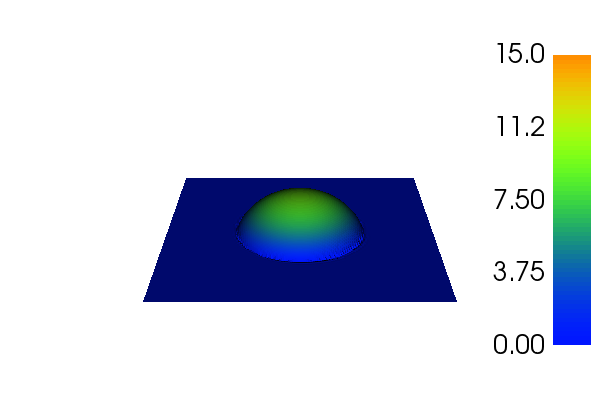}}
\subfloat[$t= 0.73$]{\includegraphics[width=0.5\textwidth, height=0.3\textheight]{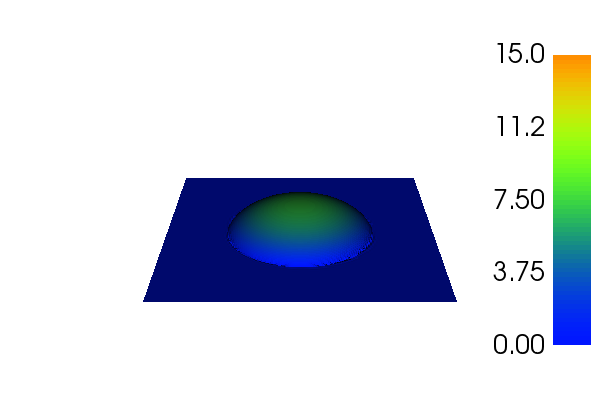}} 
\caption{Surface plot of the solution at several values of time $t$ with $m=2.5$ and $C_B= 0.005$.}
\label{fig:2}
\end{figure}

\section{Conclusion}\label{sec:concl}
We presented and analysed the gradient discretisation method for the porous me\-dium equation, in both slow and fast diffusion regimes. Using discrete functional analysis techniques, provided by the GDM framework and involving in particular a discrete Ascoli--Arzela theorem, we obtained a strong $L^2$ convergence result for the approximate gradients, and a uniform-in-time strong $L^{m+1}(\Omega)$ convergence result for the approximate solutions. These results apply to all methods that fall into the GDM framework. We illustrated the theoretical convergence using the mass-lumped conforming $\mathbb{P}_1$ and the Hybrid Mimetic Mixed methods to approximate the Barenblatt solution. The overall numerical approximations, including the location of the front, remain reasonably good in both slow and fast diffusion regimes.
\section{Appendix}\label{sec:appen}
The proofs of following lemmas can be found in~\cite{Droniou2016} (see Lemma 3.4 and 3.5 in this reference).

\begin{lemma}\label{lem: ap2}
Let $I\subset \R$ be a closed interval and $F: I\rightarrow (-\infty,+\infty]$ be a convex continuous function. Let $v\in L^2(\Omega;I)$ and $(v_n)_{n\in\N}\subset L^2(\Omega;I)$ be such that $(v_n)_{n\in\N}\rightarrow v$ weakly in $L^2(\Omega)$. Then
\[
\int_\Omega F(v)(\x)d\x \leq \liminf_{n\rightarrow\infty} \int_\Omega F(v_n)(\x)d\x.
\]
\end{lemma}

\begin{lemma}[Minty's trick]\label{lem: ap1}
Let $F\in C^0(\R)$ be a non-decreasing function. Let $(X,\mu)$ be a measurable set with finite measure and let $(u_n)_{n\in\N}\subset L^p(X)$, with $p>1$ satisfy
\begin{enumerate}
\item there exists $u\in L^p(X)$ such that  $(u_n)_{n\in\N}\rightarrow u$ weakly in $L^p(X)$,
\item $(F(u_n))_{n\in\N}\subset L^1(X)$ and there exists $v\in L^1(X)$ such that  $(F(u_n))_{n\in\N}\rightarrow u$ strongly in $L^1(X)$.
\end{enumerate}
Then $v=F(u)$ a.e. on $X$.
\end{lemma}

\medskip

\thanks{\textbf{Acknowledgement}: this research was supported by the Australian Government through the Australian Research Council's Discovery Projects funding scheme (pro\-ject number DP170100605).
}


\bibliographystyle{abbrv}
\bibliography{nonlinearPorousMedia}

\begin{thebibliography}{10}

\bibitem{Rui2017}
R.~M. Almeida, S.~N. Antontsev, and J.~C. Duque.
\newblock Discrete solutions for the porous medium equation with absorption and
  variable exponents.
\newblock {\em Mathematics and Computers in Simulation}, 137:109--129, 2017.
\newblock MAMERN VI-2015: 6th International Conference on Approximation Methods
  and Numerical Modeling in Environment and Natural Resources.

\bibitem{Fenics}
M.~Alnæs, J.~Blechta, J.~Hake, A.~Johansson, B.~Kehlet, A.~Logg,
  C.~Richardson, J.~Ring, M.~Rognes, and G.~Wells.
\newblock The {FEniCS} project version 1.5.
\newblock {\em Archive of Numerical Software}, 3(100), 2015.

\bibitem{AWZ97}
T.~Arbogast, M.~F. Wheeler, and N.-Y. Zhang.
\newblock A nonlinear mixed finite element method for a degenerate parabolic
  equation arising in flow in porous media.
\newblock {\em SIAM J. Numer. Anal.}, 33(4):1669--1687, 1996.

\bibitem{aziz2002petroleum}
K.~Aziz and A.~Settari.
\newblock {\em Petroleum Reservoir Simulation}.
\newblock Oil reservoir engineering. Applied Science Publishers, 2002.

\bibitem{Barenblatt52}
G.~I. Barenblatt.
\newblock On some unsteady motions of a liquid and gas in a porous medium.
\newblock {\em Akad. Nauk SSSR. Prikl. Mat. Meh.}, 16:67--78, 1952.

\bibitem{BerrymanHolland1978}
J.~G. Berryman and C.~J. Holland.
\newblock Nonlinear diffusion problem arising in plasma physics.
\newblock {\em Phys. Rev. Lett.}, 40:1720--1722, 1978.

\bibitem{Berryman1980}
J.~G. Berryman and C.~J. Holland.
\newblock {Stability of the separable solution for fast diffusion}.
\newblock {\em Archive for Rational Mechanics and Analysis}, 74:379--388, 1980.

\bibitem{BC17}
K.~Brenner and C.~Canc\`es.
\newblock Improving newton's method performance by parametrization: the case of
  richards equation.
\newblock {\em SIAM J. Numer. Anal.}, 55(4):1760--1785, 2017.

\bibitem{Caffarelli1980}
L.~Caffarelli and A.~Friedman.
\newblock Regularity of the free boundary of a gas flow in an n-dimensional
  porous medium.
\newblock {\em Indiana University Mathematics Journal}, 29(3):361--391, 1980.

\bibitem{CR18}
L.~Cappanera and B.~Rivi\`ere.
\newblock Discontinuous {G}alerkin method for solving the black-oil problem in
  porous media.
\newblock {\em Numer. Methods Partial Differential Equations}, 35(2):761--789,
  2019.

\bibitem{CarChaGraSwi1990}
J.~M. Carlson, E.~R. Changes, E.~R. Grannan, and G.~H. Swindle.
\newblock Self-organized criticality and singular diffusions.
\newblock {\em Phys. Rev. Lett.}, 65:2547--2550, 1990.

\bibitem{CD07}
C.~Chainais-Hillairet and J.~Droniou.
\newblock Convergence analysis of a mixed finite volume scheme for an
  elliptic-parabolic system modeling miscible fluid flows in porous media.
\newblock {\em {SIAM} {J}. {N}umer. {A}nal.}, 45(5):2228--2258 (electronic),
  2007.

\bibitem{CL71}
M.~Crandall and T.~Liggett.
\newblock Generation of semi-groups of nonlinear transformations on general
  banach spaces.
\newblock {\em Amer. J. Math.}, 93:265--298, 1971.

\bibitem{Diben1991}
E.~DiBenedetto, Y.~Kwong, and V.~Vespri.
\newblock Local space-analiticity of solutions of certain singular parabolic
  equations.
\newblock {\em Indiana University Mathematics Journal}, 40(2):741--765, 1991.

\bibitem{poly}
J.~Droniou.
\newblock Int\'egration et espaces de sobolev \`a valeurs vectorielles.
\newblock Polyco\-pi\'es de l'Ecole Doctorale de Math\'ematiques-Informatique
  de Marseille, available at \texttt{http://www-gm3.univ-mrs.fr/polys}, 2001.

\bibitem{DE06}
J.~Droniou and R.~Eymard.
\newblock A mixed finite volume scheme for anisotropic diffusion problems on
  any grid.
\newblock {\em Numer. {M}ath.}, 105(1):35--71, 2006.

\bibitem{Droniou2016}
J.~Droniou and R.~Eymard.
\newblock Uniform-in-time convergence of numerical methods for non-linear
  degenerate parabolic equations.
\newblock {\em Numerische Mathematik}, 132(4):721--766, 2016.

\bibitem{DE19}
J.~Droniou and R.~Eymard.
\newblock High-order mass-lumped schemes for nonlinear degenerate elliptic
  equations.
\newblock {\em SIAM J. Numer. Anal.}, 58(1):153--188, 2020.

\bibitem{Droniou.et.al2018}
J.~Droniou, R.~Eymard, T.~Gallou\"et, C.~Guichard, and R.~Herbin.
\newblock {\em The gradient discretisation method}, volume~82 of {\em
  Mathematics \& Applications}.
\newblock Springer, 2018.

\bibitem{DEGH09}
J.~Droniou, R.~Eymard, T.~Gallou{\"e}t, and R.~Herbin.
\newblock A unified approach to mimetic finite difference, hybrid finite volume
  and mixed finite volume methods.
\newblock {\em Math. Models Methods Appl. Sci.}, 20(2):265--295, 2010.

\bibitem{DHM16}
J.~Droniou, J.~Hennicker, and R.~Masson.
\newblock Numerical analysis of a two-phase flow discrete fracture matrix
  model.
\newblock {\em Numer. Math.}, 141(1):21--62, 2019.

\bibitem{Duque2013}
J.~Duque, R.~Almeida, and S.~Antontsev.
\newblock Convergence of the finite element method for the porous media
  equation with variable exponent.
\newblock {\em SIAM Journal on Numerical Analysis}, 51(6):3483--3504, 2013.

\bibitem{Ebmeyer1998}
C.~Ebmeyer.
\newblock Error estimates for a class of degenerate parabolic equations.
\newblock {\em SIAM Journal on Numerical Analysis}, 35(3):1095--1112, 1998.

\bibitem{Ebmeyer2008}
C.~Ebmeyer and W.~Liu.
\newblock Finite element approximation of the fast diffusion and the porous
  medium equations.
\newblock {\em SIAM Journal on Numerical Analysis}, 46(5):2393--2410, 2008.

\bibitem{EM12}
A.~Ern and I.~Mozolevski.
\newblock Discontinuous {G}alerkin method for two-component liquid-gas porous
  media flows.
\newblock {\em Comput. Geosci.}, 16(3):677--690, 2012.

\bibitem{sushi}
R.~Eymard, T.~Gallou{\"e}t, and R.~Herbin.
\newblock Discretization of heterogeneous and anisotropic diffusion problems on
  general nonconforming meshes {SUSHI}: a scheme using stabilization and hybrid
  interfaces.
\newblock {\em IMA J. Numer. Anal.}, 30(4):1009--1043, 2010.

\bibitem{eym03}
R.~Eymard, R.~Herbin, and A.~Michel.
\newblock Mathematical study of a petroleum-engineering scheme.
\newblock {\em M2AN Math. Model. Numer. Anal.}, 37(6):937--972, 2003.

\bibitem{EHV06}
R.~Eymard, D.~Hilhorst, and M.~Vohral\'{\i}k.
\newblock A combined finite volume--nonconforming/mixed-hybrid finite element
  scheme for degenerate parabolic problems.
\newblock {\em Numer. Math.}, 105(1):73--131, 2006.

\bibitem{GLR17}
V.~Girault, J.~Li, and B.~M. Rivi\`ere.
\newblock Strong convergence of the discontinuous {G}alerkin scheme for the low
  regularity miscible displacement equations.
\newblock {\em Numer. Methods Partial Differential Equations}, 33(2):489--513,
  2017.

\bibitem{IK02}
K.~Ito and F.~Kappel.
\newblock {\em Evolution equations and approximations}, volume~61 of {\em
  Series on Advances in Mathematics for Applied Sciences}.
\newblock World Scientific Publishing Co., Inc., River Edge, NJ, 2002.

\bibitem{JK91}
W.~J\"{a}ger and J.~Ka\v{c}ur.
\newblock Solution of porous medium type systems by linear approximation
  schemes.
\newblock {\em Numer. Math.}, 60(3):407--427, 1991.

\bibitem{JK95}
W.~J\"{a}ger and J.~Ka\v{c}ur.
\newblock Solution of doubly nonlinear and degenerate parabolic problems by
  relaxation schemes.
\newblock {\em RAIRO Mod\'{e}l. Math. Anal. Num\'{e}r.}, 29(5):605--627, 1995.

\bibitem{K59}
T.~Kato.
\newblock Remarks on pseudo-resolvents and infinitesimal generators of
  semi-groups.
\newblock {\em Proc. Japan Acad.}, 35:467--468, 1959.

\bibitem{K97}
J.~Ka\v{c}ur.
\newblock Solution of degenerate parabolic systems by relaxation schemes.
\newblock In {\em Proceedings of the {S}econd {W}orld {C}ongress of {N}onlinear
  {A}nalysts, {P}art 7 ({A}thens, 1996)}, volume~30, pages 4629--4636, 1997.

\bibitem{Li2018}
X.~Li.
\newblock Anisotropic mesh adaptation for finite element solution of
  anisotropic porous medium equation.
\newblock {\em Computers \& Mathematics with Applications}, 75(6):2086 -- 2099,
  2018.
\newblock 2nd Annual Meeting of SIAM Central States Section, September
  30-October 2, 2016.

\bibitem{mfdrev}
K.~Lipnikov, G.~Manzini, and M.~Shashkov.
\newblock Mimetic finite difference method.
\newblock {\em J. Comput. Phys.}, 257-Part B:1163--1227, 2014.

\bibitem{LR16}
F.~List and F.~Radu.
\newblock A study on iterative methods for solving richards' equation.
\newblock {\em Comput. Geosci.}, 20:341--353, 2016.

\bibitem{MNV87}
E.~Magenes, R.~H. Nochetto, and C.~Verdi.
\newblock Energy error estimates for a linear scheme to approximate nonlinear
  parabolic problems.
\newblock {\em RAIRO Mod\'{e}l. Math. Anal. Num\'{e}r.}, 21(4):655--678, 1987.

\bibitem{Nochetto86}
R.~H. Nochetto.
\newblock A note on the approximation of free boundaries by finite element
  methods.
\newblock {\em RAIRO Mod\'{e}l. Math. Anal. Num\'{e}r.}, 20(2):355--368, 1986.

\bibitem{NochettoVerdi88}
R.~H. Nochetto and C.~Verdi.
\newblock Approximation of degenerate parabolic problems using numerical
  integration.
\newblock {\em SIAM J. Numer. Anal.}, 25(4):784--814, 1988.

\bibitem{Pattle1959}
R.~E. Pattle.
\newblock Diffusion from an instantaneous point source with a
  concentration-dependent coefficient.
\newblock {\em The Quarterly Journal of Mechanics and Applied Mathematics},
  12(4):407--409, 1959.

\bibitem{Pop02}
I.~S. Pop.
\newblock Error estimates for a time discretization method for the {R}ichards'
  equation.
\newblock {\em Comput. Geosci.}, 6:141--160, 2002.

\bibitem{PRK04}
I.~S. Pop, F.~Radu, and P.~Knabner.
\newblock Mixed finite elements for the {R}ichards' equation: linearization
  procedure.
\newblock {\em J. Comput. Appl. Math.}, 168(1-2):365--373, 2004.

\bibitem{Raduetc17}
F.~A. Radu, K.~Kumar, J.~M. Nordbotten, and I.~S. Pop.
\newblock A robust, mass conservative scheme for two-phase flow in porous media
  including {H}\"older continuous nonlinearities.
\newblock {\em IMA Journal of Numerical Analysis}, 38(2):884--920, 06 2017.

\bibitem{RPK06}
F.~A. Radu, I.~S. Pop, and P.~Knabner.
\newblock Newton-type methods for the mixed finite element discretization of
  some degenerate parabolic equations.
\newblock In {\em Numerical mathematics and advanced applications}, pages
  1192--1200. Springer, Berlin, 2006.

\bibitem{RPK08}
F.~A. Radu, I.~S. Pop, and P.~Knabner.
\newblock Error estimates for a mixed finite element discretization of some
  degenerate parabolic equations.
\newblock {\em Numer. Math.}, 109(2):285--311, 2008.

\bibitem{RW11}
B.~M. Rivi\`ere and N.~J. Walkington.
\newblock Convergence of a discontinuous {G}alerkin method for the miscible
  displacement equation under low regularity.
\newblock {\em SIAM J. Numer. Anal.}, 49(3):1085--1110, 2011.

\bibitem{EKR04}
E.~Schneid, P.~Knabner, and F.~Radu.
\newblock A priori error estimates for a mixed finite element discretization of
  the {R}ichards' equation.
\newblock {\em Numer. Math.}, 98(2):353--370, 2004.

\bibitem{Ben07}
B.~Schweizer.
\newblock Regularization of outflow problems in unsaturated porous media with
  dry regions.
\newblock {\em Journal of Differential Equations}, 237(2):278--306, 2007.

\bibitem{Teso2018}
F.~D. Teso, J.~Endal, and E.~R. Jakobsen.
\newblock Robust numerical methods for local and nonlocal equations of porous
  medium type. {P}art {I}: theory.
\newblock {\em arXiv:1801.07148v1}, 2018.

\bibitem{T58}
H.~Trotter.
\newblock Approximations of semigroups.
\newblock {\em Pacific J. Math.}, 8:887--919, 1958.

\bibitem{Vazquez1992}
J.-L. Vazquez.
\newblock {\em An Introduction to the Mathematical Theory of the Porous Medium
  Equation}, pages 347--389.
\newblock Springer Netherlands, Dordrecht, 1992.

\bibitem{Vaz2006}
J.-L. Vazquez.
\newblock {\em Smoothing and Decay Estimates for Nonlinear Diffusion Equations.
  Equations of Porous Medium Type}, volume~33 of {\em Oxford Lecture Series in
  Mathematics and Its Applications}.
\newblock Oxford University Press, Oxford, 2006.

\bibitem{Vazquez2007}
J.-L. V\'azquez.
\newblock {\em The porous medium equation: Mathematical Theory}.
\newblock Oxford Mathematical Monographs. The Clarendon Press Oxford University
  Press, 2007.

\bibitem{Y97}
I.~Yotov.
\newblock A mixed finite element discretization on non-matching multiblock
  grids for a degenerate parabolic equation arising in porous media flow.
\newblock {\em East-West J. Numer. Math.}, 5(3):211--230, 1997.

\end{thebibliography}
\end{document}